\documentclass[11pt,a4paper]{article}
\usepackage{amsfonts}
\usepackage{graphicx,tikz}
\usepackage{epsf,epsfig,amsfonts,amsgen,amsmath,amstext,amsbsy,amsopn,amsthm,lineno}
\usepackage{color}
\usepackage{rotating}

\theoremstyle{theorem}
\newtheorem{theorem}{Theorem}[section]

\newtheorem{lemma}{Lemma}[section]
\newtheorem{corollary}{Corollary}

\newtheorem{observation}{Observation}

\theoremstyle{definition}

\newtheorem{claim}{Claim}

\newtheorem{construction}{Construction}[section]

\newcommand{\clo}{{\rm cl}^{\rm o}}
\newcommand{\clr}{{\rm cl}^{\rm r}}
\newcommand{\clc}{{\rm cl}^{\rm c}}
\date{}
\begin{document}

\title{\bf
Closures and heavy pairs for hamiltonicity\thanks{Supported by NSFC (Nos. 12071370, 12171393, and 12131013), China Scholarship Council,
P.R.~China
(No. 202206290004) and Shaanxi Fundamental Science Research Project for Mathematics and Physics (No. 22JSZ009).}}

\author{Wangyi Shang$^{a,b,c}$, Hajo Broersma$^{b,}$\thanks{Corresponding author.
 E-mail addresses: wangyishang@mail.nwpu.edu.cn (W.~Shang), h.j.broersma@utwente.nl (H.J.~Broersma), sgzhang@nwpu.edu.cn (S.~Zhang) , binlongli@nwpu.edu.cn (B.~Li).}, Shenggui Zhang$^{a,c}$, Binlong Li$^{a,c}$\\[2mm]
\small $^a$ School of Mathematics and Statistics, Northwestern Polytechnical University,\\
\small Xi'an, Shaanxi 710129, P. R.~China\\
\small $^b$ Faculty of Electrical Engineering,
Mathematics
and Computer Science, University of Twente,\\
\small P.O. Box 217, 7500 AE Enschede, The Netherlands\\
\small $^c$ Xi'an-Budapest Joint Research Center for Combinatorics, Northwestern Polytechnical University,\\
\small Xi'an, Shaanxi 710129, P. R.~China \\}
\maketitle

\begin{center}
\begin{minipage}{130mm}
\small\noindent{\bf Abstract:} We say that a graph $G$ on $n$ vertices is $\{H,F\}$-$o$-heavy if every induced subgraph of $G$ isomorphic to $H$ or $F$ contains two nonadjacent vertices with degree sum at least $n$. Generalizing earlier sufficient forbidden subgraph conditions for hamiltonicity, in 2012, Li, Ryj\'a\v{c}ek, Wang and Zhang determined all connected graphs $R$ and $S$ of order at least 3 other than $P_3$ such that every 2-connected $\{R,S\}$-$o$-heavy graph is hamiltonian. In particular, they showed that, up to symmetry, $R$ must be a claw and $S\in\{P_4,P_5,C_3,Z_1,Z_2,B,N,W\}$. In 2008, \v{C}ada extended  Ryj\'a\v{c}ek's closure concept for claw-free graphs by introducing what we call the $c$-closure for claw-$o$-heavy graphs. We apply it here to characterize the structure of the $c$-closure of 2-connected $\{R,S\}$-$o$-heavy graphs, where $R$ and $S$ are as above. Our main results extend or generalize several earlier results on hamiltonicity involving forbidden or $o$-heavy subgraphs.

\smallskip
\noindent{\bf Keywords:} Closure, heavy subgraph, hamiltonian graph, claw-free graph, claw-$o$-heavy graph
\end{minipage}
\end{center}

\smallskip

\section{Introduction}\label{intro}

We consider finite simple undirected graphs only, and for
terminology and notation not defined here we refer the reader to~\cite{BondyMurty}.

Our research is motivated by a long list of earlier work regarding different types of sufficient conditions for hamiltonicity of graphs, i.e., for the existence of a cycle containing every vertex of the graph. The earliest conditions in this area go back to the 1950s, when the first degree conditions were established. This was later followed by forbidden induced subgraph conditions, an active area since the 1980s. More recently, several authors have considered combining the two types of conditions by imposing degree conditions on induced subgraphs instead of forbidding them. In another direction, degree conditions and forbidden subgraph conditions were the motivation to introduce closure concepts for hamiltonicity. These concepts aim to allow the addition of edges to the graph without affecting the (non)hamiltonicity. We will see examples of all of the above types of results in the sequel. However, we will refrain from an exhaustive overview due to the length of the paper.

Let $G$ be a graph. For a vertex $v\in V(G)$, we use $N_{G}(v)$ (the \emph{neighborhood} of $v$) to denote the set, and $d_{G}(v)$ (the \emph{degree} of $v$) for the number, of neighbors of $v$ in $G$. For a subgraph $G'$ of $G$, we define $N_G(G')=\bigcup_{v\in V(G')} N_G(v)\setminus V(G')$. We skip the subscript $G$ in the above expressions if no confusion may arise.

We begin with the following classic sufficient degree condition for hamiltonian graphs.

\begin{theorem}[Ore \cite{Ore}]\label{ThOr}
Let $G$ be a graph on $n\geq 3$ vertices. If every two nonadjacent
vertices of $G$ have degree sum at least $n$, then $G$ is
hamiltonian.
\end{theorem}

Based on this result, we say that a pair of nonadjacent vertices in a graph $G$ is \emph{$o$-heavy} (where the $o$ refers to Ore) if their degree sum is at least $|V(G)|$. Motivated by the proof of the above result, Bondy and Chv\'atal~\cite{BondyChvatal} introduced the following closure concept. The \emph{o-closure} of a graph $G$, denoted by $\clo(G)$, is the graph obtained from $G$ by recursively joining $o$-heavy pairs of $G$ by an edge (and adapting the degrees of the vertices of $G$), till there are no such pairs left. Bondy and Chv\'atal showed that the $o$-closure of a graph is uniquely determined, and proved the following result.

\begin{theorem}[Bondy and Chv\'atal \cite{BondyChvatal}]\label{ThBoCh}
A graph $G$ is hamiltonian if and only if $\clo(G)$ is hamiltonian.
\end{theorem}

Next, we turn to forbidden induced subgraph conditions. For a set $S\subseteq V(G)$,
the subgraph $G[S]$ of $G$ \emph{induced by} $S$ has vertex set $S$ and all edges of $G$ which join pairs of vertices of $S$. A graph $H$ is an \emph{induced subgraph} of $G$ if $H$ is isomorphic to $G[S]$ for some set $S\subseteq V(G)$. Let $H$ be a given graph.
Then a graph $G$ is called \emph{$H$-free} if $G$ contains no induced subgraph isomorphic to $H$. If we impose that $G$ is $H$-free, then we also say that $H$ is a \emph{forbidden subgraph} of $G$. Note that if $H_1$ is an induced subgraph of $H_2$, then an $H_1$-free graph is also $H_2$-free. For a family $\mathcal{H}$ of graphs, $G$ is said to be \emph{$\mathcal{H}$-free} if $G$ is $H$-free for every $H\in\mathcal{H}$.

The graph $K_{1,3}$ is usually called a \emph{claw}, and we also use the more common term \emph{claw-free} instead of $K_{1,3}$-free. By $P_k$ ($k\geq 1$) and $C_k$ ($k\geq 3$) we denote a path and cycle on $k$ vertices, respectively. We write $Z_k$ ($k\geq 1$) for the graph obtained by identifying a vertex of
a $C_3$ with an end-vertex of a $P_{k+1}$. Some other special graphs we will use in the paper are shown in Figure~\ref{fig1}. We refer to these graphs by the capital letters $B$, $N$, $W$ and $D$, or by their common names \emph{bull}, \emph{net}, \emph{wounded} and \emph{diamond}, respectively.

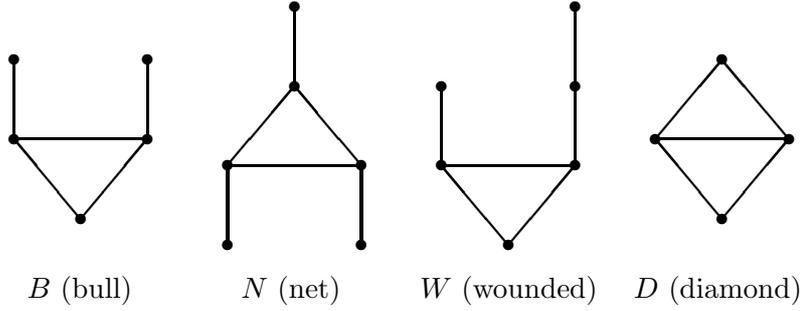
\begin{figure}[h]
\centering
\begin{picture}(320,130)
\thicklines

\put(0,0){\put(45,40){\circle*{4}} \put(45,40){\line(-5,6){25}}
\put(45,40){\line(5,6){25}} \put(20,70){\line(1,0){50}}
\multiput(20,70)(50,0){2}{\multiput(0,0)(0,30){2}{\put(0,0){\circle*{4}}}
\put(0,0){\line(0,1){30}}} \put(25,10){$B$ (bull)}}

\put(80,0){\multiput(20,30)(50,0){2}{\put(0,0){\line(0,1){30}}
\multiput(0,0)(0,30){2}{\circle*{4}}}
\multiput(45,90)(0,30){2}{\put(0,0){\circle*{4}}}
\put(45,90){\line(0,1){30}} \put(20,60){\line(1,0){50}}
\put(20,60){\line(5,6){25}} \put(70,60){\line(-5,6){25}}
\put(25,10){$N$ (net)}}

\put(160,0){\put(45,30){\circle*{4}} \put(20,60){\line(1,0){50}}
\put(45,30){\line(5,6){25}} \put(45,30){\line(-5,6){25}}
\multiput(20,60)(0,30){2}{\put(0,0){\circle*{4}}}
\multiput(70,60)(0,30){3}{\put(0,0){\circle*{4}}}
\put(20,60){\line(0,1){30}} \put(70,60){\line(0,1){60}}
\put(12,10){$W$ (wounded)}}

\put(240,0){\put(45,40){\circle*{4}} \put(20,70){\circle*{4}}
\put(70,70){\circle*{4}} \put(45,100){\circle*{4}}
\put(45,40){\line(-5,6){25}} \put(45,40){\line(5,6){25}}
\put(20,70){\line(5,6){25}} \put(70,70){\line(-5,6){25}}
\put(20,70){\line(1,0){50}} \put(12,10){$D$ (diamond)}}

\end{picture}\caption{\small The graphs $B$, $N$, $W$ and $D$.}
\label{fig1}
\end{figure}

By considering pairs of forbidden subgraphs, Duffus et al. in~\cite{DuGoJa} obtained the following result.

\begin{theorem}[Duffus et al. \cite{DuGoJa}]\label{ThBuGoJa}
Every 2-connected $\{K_{1,3},N\}$-free graph is hamiltonian.
\end{theorem}

A natural question is which pairs of (connected) forbidden subgraphs lead to the same conclusion. Clearly, forbidding a graph on one vertex or on one edge makes little sense, and forbidding a $P_3$ (in a connected graph) yields a complete graph. Taking this into account, Bedrossian~\cite{Bedrossian} characterized all pairs of forbidden subgraphs for
hamiltonicity, in the following sense.

\begin{theorem}[Bedrossian \cite{Bedrossian}]\label{ThBe}
Let $R,S$ be connected graphs of order at least 3 with $R,S\neq P_3$,
and let $G$ be a 2-connected graph. Then $G$ being $\{R,S\}$-free
implies $G$ is hamiltonian if and only if (up to symmetry)
$R=K_{1,3}$ and $S=P_4$, $P_5$, $P_6$, $C_3$, $Z_1$, $Z_2$, $B$, $N$ or
$W$.
\end{theorem}

To research the hamiltonian properties of claw-free graphs, Ryj\'a\v{c}ek~\cite{Ryjacek} introduced a very elegant closure concept for claw-free graphs. We will recall its definition in Section~\ref{closures}.
He showed that this closure, which we denote by $\clr(G)$, is uniquely determined
and claw-free, and that a claw-free graph $G$ is hamiltonian if and only if $\clr(G)$ is hamiltonian. A claw-free graph $G$ for which $G=\clr(G)$ will be called \emph{$r$-closed}.

Using this closure concept, it was shown in another paper by Ryj\'a\v{c}ek~\cite{Ryjacek'} that the list of forbidden subgraphs for $S$ in the characterization given in Theorem~\ref{ThBe} can be reduced to just one graph, namely the net. Moreover, the structure of $r$-closed $\{K_{1,3},N\}$-free graphs was fully determined in~\cite{Ryjacek'}.

\begin{theorem}[Ryj\'a\v{c}ek \cite{Ryjacek'}]\label{ThRy'}
Let $G$ be a 2-connected $\{K_{1,3},S\}$-free graph, where
$S\in\{P_4,P_5,P_6,C_3,Z_1,\\Z_2,B,N,W\}$. Then $\clr(G)$ is $\{K_{1,3},N\}$-free.
\end{theorem}

Note that using the closure of Ryj\'a\v{c}ek,  Theorems~\ref{ThBuGoJa} and \ref{ThRy'} give a new proof for the `if' part of Theorem~\ref{ThBe}. To recall the structural characterization of $r$-closed $\{K_{1,3},N\}$-free graphs given in~\cite{Ryjacek'}, we first need to describe the following two classes of graphs. We also need these classes in Section~\ref{c-closed}.

\begin{construction}
Denote by $\mathcal{C}_1^N$ the class of graphs illustrated in Figure~\ref{C_1^N} and
obtained by the following construction:
(i) for an integer $t\ge 1$, take $t$ distinct complete graphs $K^1,K^2,\ldots, K^t$ with $|V(K^i)|\geq 4$ for $2\leq i \leq t-1$ (if $t\geq 3$) and $|V(K^i)|\geq 2$ for $i=1,t$;
(ii) choose subsets $U^1_2\subset V(K^1)$ and $U^t_1\subset V(K^t)$ such that $|U^1_2|,|U^t_1|\geq 2$;
(iii) choose two disjoint subsets $U^i_1, U^i_2\subset V(K^i)$, $i=2,\ldots,t-1$,
     such that $|U^i_1|,|U^i_2|\geq 2$ and $|U^i_2|=|U^{i+1}_1|$ for $i=1,2,\ldots,t-1$; and (iv) join the vertices of $U^i_2$ and $U^{i+1}_1$ with a perfect matching for $i=1,2,\ldots,t-1$. Note that we do not require $|U^i_1|=|U^i_2|$.
\end{construction}

\begin{figure}[h]
\centering
\begin{tikzpicture}[scale=0.5]
\draw[fill=gray!30](0,0.15)ellipse(2);
\draw(1,0.15)ellipse(0.6 and 1);
\path(1,0.8) coordinate  (a_1);\draw [fill=black] (a_1) circle (0.1cm);
\path(1,0.5) coordinate  (a);\draw [fill=black] (a) circle (0.1cm);
\draw [fill=black] (1,0.2) circle (0.05cm);
\draw [fill=black] (1,0) circle (0.05cm);
\draw [fill=black] (1,-0.2) circle (0.05cm);
\path(1,-0.5) coordinate  (c);\draw [fill=black] (c) circle (0.1cm);

\draw[fill=gray!30](5,0.15)ellipse(2);
\draw(4,0.15)ellipse(0.6 and 1);
\draw(6,0.15)circle(0.6 and 1.2);
\path(6,0.6) coordinate  (d);\draw [fill=black] (d) circle (0.1cm);
\draw [fill=black] (6,1.05) circle (0.1cm);
\draw [fill=black] (6,0.2) circle (0.05cm);
\draw [fill=black] (6,0) circle (0.05cm);
\draw [fill=black] (6,-0.2) circle (0.05cm);
\path(6,-0.7) coordinate  (f);
\draw [fill=black] (f) circle (0.1cm);

\path(4,0.5) coordinate  (g);\draw [fill=black] (g) circle (0.1cm);
\path(4,0.8) coordinate  (a_2);\draw [fill=black] (a_2) circle (0.1cm);
\draw [fill=black] (4,0.2) circle (0.05cm);
\draw [fill=black] (4,0) circle (0.05cm);
\draw [fill=black] (4,-0.2) circle (0.05cm);
\path(4,-0.5) coordinate  (i);\draw [fill=black] (i) circle (0.1cm);
\draw [line width=0.8] (a)--(g);
\draw [line width=0.8] (a_1)--(a_2);
\draw [line width=0.8] (c)--(i);

 \draw[fill=gray!30](10,0.15)ellipse(2);
 \draw(9,0.15)circle(0.6 and 1.2);
 \draw(11,0.15)circle(0.5 and 0.9);

\path(9,0.6) coordinate  (j);\draw [fill=black] (j) circle (0.1cm);
\draw [fill=black] (9,1.05) circle (0.1cm);
 \draw [fill=black] (9,0.2) circle (0.05cm);
\draw [fill=black] (9,0) circle (0.05cm);
\draw [fill=black] (9,-0.2) circle (0.05cm);
\path(9,-0.7) coordinate  (l);\draw [fill=black] (l) circle (0.1cm);

\draw [line width=0.8] (d)--(j);
\draw [line width=0.8] (f)--(l);
\draw [line width=0.8] (6,1.05)--(9,1.05);

\path(11,0.5) coordinate  (m);\draw [fill=black] (m) circle (0.1cm);
\draw [fill=black] (11,0.85) circle (0.1cm);
 \draw [fill=black] (11,0.2) circle (0.05cm);
\draw [fill=black] (11,0) circle (0.05cm);
\draw [fill=black] (11,-0.2) circle (0.05cm);
\path(11,-0.5) coordinate  (p);\draw [fill=black] (p) circle (0.1cm);

\draw [line width=0.8] (m)--(13,0.5);
\draw [line width=0.8] (p)--(13,-0.4);

\draw [line width=0.8] (11,0.85)--(13,0.85);
\draw [fill=black] (13.5,0) circle (0.05cm);
\draw [fill=black] (14,0) circle (0.05cm);
\draw [fill=black] (14.5,0) circle (0.05cm);

\draw[fill=gray!30](18,0.15)ellipse(2);
\draw(17,0.2)circle(0.6 and 1.4);
\path(17,0.6) coordinate  (t);\draw [fill=black] (t) circle (0.1cm);
\draw [fill=black] (17,1.15) circle (0.1cm);
\draw [fill=black] (17,0.2) circle (0.05cm);
\draw [fill=black] (17,0) circle (0.05cm);
\draw [fill=black] (17,-0.2) circle (0.05cm);
\path(17,-0.6) coordinate  (x);\draw [fill=black] (x) circle (0.1cm);

\draw [line width=0.8] (15,0.6)--(t);
\draw [line width=0.8] (15,-0.6)--(x);
\draw [line width=0.8] (15,1.15)--(17,1.15);

\end{tikzpicture}
\caption{\small The class of graphs $\mathcal{C}_1^N$.}\label{C_1^N}
\end{figure}
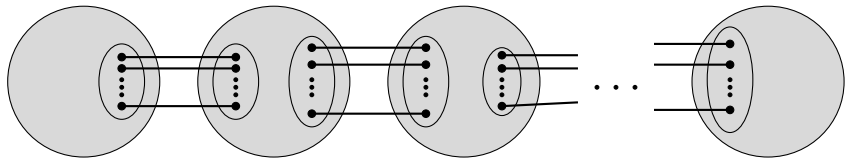

\begin{construction}
Denote by $\mathcal{C}_2^N$ the class of graphs illustrated in Figure~\ref{C_2^N} and obtained by the following construction:
(i) for an integer $t\ge 3$, take $t$ distinct complete graphs $K^1,K^2,\ldots, K^t$ with ${|V(K^i)|}\geq 2$ for $1\leq i \leq t$;
(ii) choose two nonempty disjoint subsets $U^i_1, U^i_2\subset V(K^i)$ such that $|U^i_2|=|U^{i+1}_1|$ for $i=1,2,\ldots,t$ (superscripts modulo $t$; for $t=3$ furthermore assume there exists an integer $j$, $1\leq j\leq 3$ such that $|U^j_1|\geq 2$); and
(iii) for $i=1,2,\ldots,t$ identify $U^i_2$ with $U^{i+1}_1$ if $|U^i_2|=|U^{i+1}_1|=1$ and join the vertices of $U^i_2$ and $U^{i+1}_1$ with a perfect matching if $|U^i_2|=|U^{i+1}_1|\geq 2$, respectively.
\end{construction}

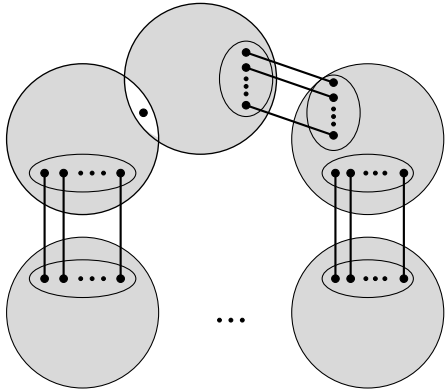
\begin{figure}[h]
\centering
\begin{tikzpicture}[scale=0.5]
\draw[fill=gray!30](0,-2.3)circle(2);
\draw[fill=gray!30](7.5,2.3)circle(2);
\draw[fill=gray!30](7.5,-2.3)circle(2);
\draw(0,-1.4)circle(1.4 and 0.5);

\def\r{2cm}
\fill[gray!30] (0,2.3) circle[radius=\r];
\fill[gray!30] (3.1,3.9) circle[radius=\r];
\begin{scope}
\clip (0,2.3) circle[radius=\r];
\fill[white] (1.3*\r,3.69) circle[radius=0.74*\r];
\end{scope}
\draw[line width=0.5pt] (0,2.3)circle[radius=\r] (1.7*\r,0);
\draw[line width=0.5pt] (3.1,3.9)circle[radius=\r] (1.7*\r,0);
\draw(4.3,3.9)circle(0.7 and 1);
\draw(6.6,3)circle(0.7 and 1);
\draw(0,1.4)circle(1.4 and 0.5);
\path(-1,1.4) coordinate  (a);\draw [fill=black] (a) circle (0.1cm);
\path(-0.5,1.4) coordinate  (b);\draw [fill=black] (b) circle (0.1cm);
\path(1,1.4) coordinate  (c);\draw [fill=black] (c) circle (0.1cm);
\draw [fill=black] (-0.05,1.4) circle (0.05cm);
\draw [fill=black] (0.25,1.4) circle (0.05cm);
\draw [fill=black] (0.55,1.4) circle (0.05cm);
\draw [fill=black] (1.6,3) circle (0.1cm);

\path(-1,-1.4) coordinate  (d);\draw [fill=black] (d) circle (0.1cm);
\path(-0.5,-1.4) coordinate  (e);\draw [fill=black] (e) circle (0.1cm);
\path(1,-1.4) coordinate  (f);\draw [fill=black] (f) circle (0.1cm);

\draw [fill=black] (-0.05,-1.4) circle (0.05cm);
\draw [fill=black] (0.25,-1.4) circle (0.05cm);
\draw [fill=black] (0.55,-1.4) circle (0.05cm);
\draw [line width=0.8] (a)--(d);
\draw [line width=0.8] (b)--(e);
\draw [line width=0.8] (c)--(f);

\path(4.3,4.6) coordinate  (g);\draw [fill=black] (g) circle (0.1cm);
\draw [fill=black] (4.3,4.2) coordinate  (m);\draw [fill=black] (m) circle (0.1cm);
\draw [fill=black] (4.3,3.9) circle (0.05cm);
\draw [fill=black] (4.3,3.7) circle (0.05cm);
\draw [fill=black] (4.3,3.5) circle (0.05cm);
\path(4.3,3.2) coordinate  (h);\draw [fill=black] (h) circle (0.1cm);

\path(6.6,3.8) coordinate  (i);\draw [fill=black] (i) circle (0.1cm);

\draw [fill=black] (6.6,3.1) circle (0.05cm);
\draw [fill=black] (6.6,2.9) circle (0.05cm);
\draw [fill=black] (6.6,2.7) circle (0.05cm);

\path(6.6,3.4) coordinate  (n);\draw [fill=black] (n) circle (0.1cm);
\path(6.6,2.4) coordinate  (l);\draw [fill=black] (l) circle (0.1cm);
\draw [line width=0.8] (g)--(i);
\draw [line width=0.8] (h)--(l);
\draw [line width=0.8] (m)--(n);

\draw(7.5,1.4)circle(1.2 and 0.5);

\draw(7.5,-1.4)circle(1.2 and 0.5);

\path(6.65,1.4) coordinate  (o);\draw [fill=black] (o) circle (0.1cm);
\path(7.05,1.4) coordinate  (p);\draw [fill=black] (p) circle (0.1cm);
\path(8.45,1.4) coordinate  (q);\draw [fill=black] (q) circle (0.1cm);
\path(6.65,-1.4) coordinate  (r);\draw [fill=black] (r) circle (0.1cm);
\path(7.05,-1.4) coordinate  (s);\draw [fill=black] (s) circle (0.1cm);
\path(8.45,-1.4) coordinate  (t);\draw [fill=black] (t) circle (0.1cm);
\draw [fill=black] (7.45,1.4) circle (0.05cm);
\draw [fill=black] (7.7,1.4) circle (0.05cm);
\draw [fill=black] (7.95,1.4) circle (0.05cm);
\draw [fill=black] (7.45,-1.4) circle (0.05cm);
\draw [fill=black] (7.7,-1.4) circle (0.05cm);
\draw [fill=black] (7.95,-1.4) circle (0.05cm);

\draw [line width=0.8] (o)--(r);
\draw [line width=0.8] (p)--(s);
\draw [line width=0.8] (q)--(t);

\draw [fill=black] (4.2,-2.5) circle (0.05cm);
\draw [fill=black] (3.9,-2.5) circle (0.05cm);
\draw [fill=black] (3.6,-2.5) circle (0.05cm);

\end{tikzpicture}
\caption{\small The class of graphs $\mathcal{C}_2^N$.}\label{C_2^N}
\end{figure}

Apart from possibly some small exceptional graphs, the two classes of graphs we just described are precisely all  $r$-closed
$\{K_{1,3},N\}$-free graphs.

\begin{theorem}[Ryj\'a\v{c}ek \cite{Ryjacek'}]\label{Ryjacek'N}
Let $G$ be a graph of order $n\geq 10$. Then $G$ is a 2-connected $r$-closed $\{K_{1,3},N\}$-free graph if and only if $G\in \mathcal{C}_1^N\cup  \mathcal{C}_2^N$.
\end{theorem}

We now turn to a common relaxation of degree conditions and forbidden subgraph conditions. Instead of imposing degree conditions on all nonadjacent pairs of vertices or forbidding certain subgraphs, the idea is to impose degree conditions on some nonadjacent pairs of vertices of such subgraphs if they appear as induced subgraphs in the graph. Before we can give more details, we first need some additional terminology and notation.

Let $G$ be a graph. We say that a vertex $u\in V(G)$ is \emph{heavy} in $G$ if $d(u)\geq \frac{|V(G)|}{2}$; a pair of vertices $\{u,v\}$  is an \emph{$a$-heavy pair} (or \emph{adjacent heavy pair}) of $G$ if $d(u)+d(v)\geq |V(G)|$ and $uv\in E(G)$. Note that if $\{u,v\}$ is an $o$-heavy or $a$-heavy pair, then at least one of $u$ and $v$ is a heavy vertex. Also note that Ore's Theorem (Theorem~\ref{ThOr}) states that $G$ is
hamiltonian if $|V(G)|\ge 3$ and every pair of nonadjacent vertices of $G$ is an
$o$-heavy pair.

Let $G'$ be an induced subgraph of $G$. If $G'$ contains an $o$-heavy pair of $G$, then we say that $G'$ is an \emph{$o$-heavy subgraph} of $G$ (or $G'$ is \emph{$o$-heavy} in $G$). For a given graph $H$, the graph $G$ is \emph{$H$-$o$-heavy} if every induced subgraph of $G$ isomorphic to $H$ is $o$-heavy in $G$. Note that an $H$-free graph is trivially $H$-$o$-heavy, and if $H_1$ is an induced subgraph of $H_2$, then an $H_1$-$o$-heavy graph is also $H_2$-$o$-heavy. For a family $\mathcal{H}$ of graphs, $G$ is \emph{$\mathcal{H}$-$o$-heavy} if $G$ is $H$-$o$-heavy for every $H\in\mathcal{H}$.

Li et al.~\cite{LiRyjacekWangZhang} completely characterized all pairs of $o$-heavy subgraphs for a 2-connected graph to be hamiltonian, thereby extending Bedrossian's result (Theorem~\ref{ThBe}).

\begin{theorem}[Li et al.~\cite{LiRyjacekWangZhang}]\label{ThLiRyWaZh}
Let $R$ and $S$ be connected graphs of order at least 3 with $R,S\neq P_3$, and let $G$ be a 2-connected graph. Then $G$ being $\{R,S\}$-$o$-heavy implies $G$ is hamiltonian if and only if (up to symmetry) $R=K_{1,3}$ and $S=P_4$, $P_5$, $C_3$, $Z_1$, $Z_2$, $B$, $N$ or $W$.
\end{theorem}

Note that the only graph that appears in the list for $S$ in Theorem~\ref{ThBe} but
misses here is $P_6$. This is for a good reason. In~\cite{LiRyjacekWangZhang}, the authors presented a class of $\{K_{1,3},P_6\}$-$o$-heavy non-hamiltonian graphs.

Motivated by Ryj\'a\v{c}ek's closure theory for claw-free graphs, \v{C}ada~\cite{Cada} proposed a closure theory for claw-$o$-heavy graphs which we will recall in Section~\ref{closures}. He showed that his newly introduced closure, which we will refer to as the \emph{$c$-closure}, is uniquely determined and preserves the length of a longest path and cycle, so also preserves hamiltonicity.

It is not difficult to see that the $o$-closure of a graph $G$ is the
minimum spanning supergraph of $G$ with no $o$-heavy pairs (see~\cite{BondyChvatal}),
and that the $r$-closure of a claw-free graph $G$ is the minimum spanning
supergraph of $G$ with no induced diamond (see Figure~\ref{fig1} and \cite{Ryjacek}).
As we will show in the next section, the $c$-closure of a claw-$o$-heavy graph $G$ is the minimum spanning supergraph of $G$ with no $o$-heavy pair and no induced diamond.

A nice feature of \v{C}ada's closure theory is that if $G$ is claw-$o$-heavy, then the $c$-closure of $G$ is claw-free. Since Ryj\'a\v{c}ek showed that the $r$-closure of a
2-connected $\{K_{1,3},N\}$-free graph is $\{K_{1,3},N\}$-free (see Theorem~\ref{ThRy'}), one may ask whether the $c$-closure of $G$ is also $N$-free if $G$ is $\{K_{1,3},N\}$-$o$-heavy? The answer is negative.
However, as we will show in Section~\ref{heavy} of this paper,  the $c$-closure of any 2-connected $\{K_{1,3},S\}$-$o$-heavy graph is claw-free and $N$-$pq$-heavy (to be defined later), where $S=P_4$, $P_5$, $C_3$, $Z_1$, $Z_2$, $B$, $N$ or
$W$.

We recall the two relevant closures and the related structural results that we need in our later proofs in Section~\ref{pre}. The proofs of our main structural results are postponed to Section~\ref{proofs}. Moreover, as a counterpart of Ryj\'a\v{c}ek's structural characterization of 2-connected $r$-closed $\{K_{1,3},N\}$-free graphs (Theorem~\ref{Ryjacek'N}), in Section~\ref{c-closed} we fully describe the structure of 2-connected, $c$-closed, claw-free and $N$-$pq$-heavy graphs.

\section{Some preliminaries}\label{pre}
In this section, we recall the two closure theories developed by Ryj\'a\v{c}ek and \v{C}ada, respectively (which we call the $r$-closure and the $c$-closure, respectively), and present some useful lemmas and theorems which we need in the proofs of our results.

\subsection{Closures for claw-free and claw-$o$-heavy graphs}\label{closures}

During a workshop that was held in hotel the H\"olterhof near Enschede in 1995~\cite{Eidma}, Ryj\'a\v{c}ek came up with a closure concept for claw-free graphs which became known as Ryj\'a\v{c}ek's closure. The details and proofs were later published in~\cite{Ryjacek}. We give a short summary for later reference.

For this purpose, a vertex $x$ of a graph $G$ is called \emph{r-eligible} in $G$ if $x$ has a connected non-complete neighborhood in $G$. The \emph{local {$r$}-completion of $G$ at $x$}, denoted by $G'_x$, is the graph with vertex set $V(G)$ and edge set
$E(G)\cup\{uv: u,v\in N(x)\}$. It was proved in~\cite{Ryjacek} that if $G$ is claw-free, then so is $G'_x$, and if $x$ is $r$-eligible, then $G$ is hamiltonian if and only if $G'_x$ is hamiltonian. The \emph{$r$-closure} of a claw-free graph $G$, denoted by $\clr(G)$, is the graph obtained by recursively performing the local $r$-completion
operation at $r$-eligible vertices as long as this is possible.

Ryj\'a\v{c}ek~\cite{Ryjacek} proved the following result.

\begin{theorem}[Ryj\'a\v{c}ek \cite{Ryjacek}]\label{ThRy}
Let $G$ be a claw-free graph. Then: \\
(1) $\clr(G)$ is uniquely determined;\\
(2) $\clr(G)$ is the line graph of a $C_3$-free graph; and \\
(3) $\clr(G)$ is hamiltonian if and only if $G$ is hamiltonian.
\end{theorem}

Motivated by Ryj\'a\v{c}ek's closure theory for claw-free graphs,
and inspired by closure concepts in \cite{HJB} and \cite{BandT},
\v{C}ada~\cite{Cada} proposed a closure theory for claw-$o$-heavy
graphs. Let $G$ be a claw-$o$-heavy graph on $n$ vertices, and let $x\in V(G)$. Following~\cite{Cada}, we let  $E^{BC}_x= \{uv \mid u, v \in N(x), uv \notin E(G), d(u) + d(v) \ge n\}$. Let $G^{BC}_x$ be the graph with vertex set $V(G)$ and edge set $E(G) \cup E^{BC}_x$. Now we call $x$ a \emph{$c$-eligible} vertex of $G$ if $N(x)$ is not a clique in $G^{BC}_x$ and one of the followings is true:\\
(i) $G^{BC}_x[N(x)]$ is connected; or\\
(ii) $G^{BC}_x[N(x)]$ is the disjoint union of two cliques $C_1$ and
$C_2$, and $x$ is contained in an $o$-heavy pair $\{x,z\}$ of $G$ such
that $zy_1,zy_2\in E(G)$ for some $y_1\in C_1$ and $y_2\in C_2$.

Let $x$ be a $c$-eligible vertex of $G$. Then the graph
$G_x^o$ with vertex set $V(G)$ and edge set $E(G)\cup \{uv \mid u, v \in
N(x), uv \notin E(G)\}$  is called the \emph{local $c$-completion of $G$ at $x$}.

The closure concept of \v{C}ada is based on the following lemma.

\begin{lemma}[\v{C}ada \cite{Cada}]\label{LeComplation}
Let $G$ be a claw-$o$-heavy graph, and let $x$ be a $c$-eligible vertex of $G$. Then: \\
(1) for every vertex $y\in N(x)$, $d_{G_x^o}(y)\geq d_{G_x^o}(x)$;\\
(2) the graph $G_x^o$ is claw-$o$-heavy; and\\
(3) $G$ is hamiltonian if and only if $G_x^o$ is hamiltonian.
\end{lemma}

Based on this lemma, the \emph{$c$-closure} of a claw-$o$-heavy graph $G$, denoted by
$\clc(G)$, is the graph obtained by recursively performing the local {$c$}-completion operation at $c$-eligible vertices as long as this is possible. In other words, $\clc(G)$ is defined by a sequence of graphs $G_1,G_2,\ldots,G_t$, and vertices $x_1,x_2\ldots,x_{t-1}$ such that:\\
(i) $G_1=G$, $G_t=\clc(G)$; \\
(ii) $x_i$ is a $c$-eligible vertex of $G_i$, $G_{i+1}=(G_i)^o_{x_i}$, $1\leq i\leq t-1$; and\\
(iii) $\clc(G)$ has no $c$-eligible vertices.

\v{C}ada \cite{Cada}  proved the following result.

\begin{theorem}[\v{C}ada \cite{Cada}]\label{ThCa}
Let $G$ be a claw-$o$-heavy graph. Then: \\
 (1) $\clc(G)$ is uniquely determined;\\
 (2) $\clc(G)$ is the line graph of a $C_3$-free graph; and\\
 (3) $G$ is hamiltonian if and only if $\clc(G)$ is hamiltonian.
\end{theorem}

Clearly, (2) implies that the $c$-closure $\clc(G)$ of a claw-$o$-heavy graph $G$ is claw-free. In the next section, we will present some other useful structural results regarding claw-$o$-heavy graphs.

\subsection{Some useful lemmas}

Let $G$ be a claw-$o$-heavy graph, and let $C$ be a maximal clique of
$\clc(G)$. Following~\cite{LiNing}, we call $G[C]$ a \emph{region}
of $G$. For a vertex $v$ of $G$, we call $v$ an \emph{interior
vertex} if it is contained in only one region, and a \emph{frontier
vertex} if it is contained in two distinct regions. Two vertices are
called \emph{associated} in $G$ if they are in a common region, and
\emph{dissociated} otherwise. Note that two vertices are associated in $G$
if and only if they are adjacent in $\clc(G)$.

The following lemma was proved in \cite{LiNing}.

\begin{lemma}[Li and Ning \cite{LiNing}]\label{LeRegion}
Let $G$ be a claw-$o$-heavy graph, and let $R$ be a region of $G$. Then: \\
(1) $R$ is nonseparable;\\
(2) if $v$ is a frontier vertex in $R$, then $v$ has an interior
neighbor in $R$, or $R$ is complete and contains no interior vertex;\\
(3) if $u$ is associated with two vertices $u',u''\in V(R)$, then $u\in V(R)$; and\\
(4) for any two vertices $u,v\in V(R)$, there is an induced path of $G$
from $u$ to $v$ such that every internal vertex of the path is an
interior vertex in $R$.
\end{lemma}

We next prove the following useful result. Recall that $D$ is used to denote the diamond (see Figure~\ref{fig1}).

\begin{lemma}\label{LeSupergraph}
Let $G$ be a claw-$o$-heavy graph, and let $x$ be a c-eligible vertex of
$G$. If $G'$ is a $D$-free spanning supergraph of $G$ with no
$o$-heavy pair, then $G^o_x\subseteq G'$.
\end{lemma}

\begin{proof}
Clearly $G\subseteq G'$. Now let $uv$ be an arbitrary edge in
$E(G^o_x)\setminus E(G)$. We will prove that $uv\in E(G')$. Note
that $u,v\in N(x)$. Let $n=|V(G)|$.

If $uv\in E(G^{BC}_x)$,  then $d(u)+d(v)\geq n$, and
$d_{G'}(u)+d_{G'}(v)\geq n$. If $uv\notin E(G')$, then $\{u,v\}$ is
an $o$-heavy pair of $G'$, a contradiction. So we have that $uv\in
E(G')$. This implies that
$G^{BC}_x\subseteq G'$.
Now we assume that
$uv\not\in E(G^{BC}_x)$.
This implies that $uv \in E(G^o_x)\setminus E(G^{BC}_x)$.

Suppose first that
$G^{BC}_x[N(x)]$ is connected.
Let $P=uv_1v_2\cdots
v_k$ be a path in
$G^{BC}_x[N(x)]$
from $u$ to $v=v_k$. Note that
$uv_1\in
E(G^{BC}_x)
\subseteq E(G')$. If $uv\notin E(G')$, then $k\geq
2$. Let $v_i$ be the first vertex on $P-u$ such that $uv_i\notin
E(G')$. Then $uv_{i-1},v_{i-1}v_i\in E(G')$, and the subgraph of
$G'$ induced by $\{x,u,v_{i-1},v_i\}$ is a diamond, a contradiction.
Thus we conclude that $uv\in E(G')$.

For the final case,
we suppose that
$G^{BC}_x[N(x)]$
 is the disjoint union of two
cliques $C_1$ and $C_2$, and $x$ is contained in an $o$-heavy pair
$\{x,z\}$ of $G$ such that $zy_1,zy_2\in E(G)$ for some $y_1\in C_1$
and $y_2\in C_2$. Note that $xz\in E(G')$; otherwise $\{x,z\}$ is an
$o$-heavy pair of $G'$. By our assumption that $uv\notin
E(G^{BC}_x)$,
we
have (up to symmetry) $u\in C_1$ and $v\in C_2$. We claim that
$uz\in E(G')$. Suppose not. Then $u\neq y_1$, and the subgraph of
$G'$ induced by $\{x,u,y_1,z\}$ is a diamond, a contradiction. Thus,
as we claimed $uz\in E(G')$, and similarly $vz\in E(G')$. Then the
subgraph of $G'$ induced by $\{x,u,z,v\}$ is a diamond, also a
contradiction.
\end{proof}

It is not difficult to check that the $c$-closure of a claw-$o$-heavy graph
is $\{K_{1,3},D\}$-free and has no $o$-heavy pairs. By using Lemma~\ref{LeSupergraph} repeatedly, we
obtain the following result.

\begin{lemma}\label{LeClosure}
Let $G$ be a claw-$o$-heavy graph. Then $\clc(G)$ is the minimum
$\{K_{1,3},D\}$-free spanning supergraph of $G$ with no $o$-heavy
pair.
\end{lemma}

From the above Lemma \ref{LeClosure}, we can infer that any two distinct heavy vertices are adjacent in $\clc(G)$. Hence, the set of all heavy vertices of $\clc(G)$ is a clique.

A claw-$o$-heavy graph $G$ for which $G=\clc(G)$ will be called \emph{$c$-closed}.

We next prove a result on the clique structure of $c$-closed graphs.

\begin{lemma}\label{LEHeavy} Let $G$ be a $c$-closed, claw-free graph, and let $K$ be a maximal clique containing all heavy vertices. If $K$ does not contain all $a$-heavy pairs of $G$, then there exists another maximal clique $K'$ such that $K$ intersects $K'$ at {exactly one}
 heavy vertex of $G$ and $K\cup K'$ contains all $a$-heavy pairs.
\end{lemma}

\begin{proof}
Let $K'$ be a maximal clique other than $K$ containing as many $a$-heavy pairs of $G$ as possible. Since $K$ does not contain all $a$-heavy pairs of $G$, $K'$ contains at least {one}
 $a$-heavy pair of $G$. Let $\{u,v\}$ be an $a$-heavy pair of $G$ contained in $K'$. We may assume that $u$ is heavy in $G$. Then $u\in V(K)$ by Lemma~\ref{LeClosure}. It follows from claw-freeness of $G$ that $\{u\}= K\cap K'$.

Assume that there exists an $a$-heavy pair $\{x,y\}$ of $G$ such that $\{x,y\}\not\subset K\cup K'$. Let $K''$ be a maximal clique containing $\{x,y\}$. Then $K''\neq K$ and $K''\neq K'$. We may assume that $x$ is heavy in $G$. Then $K''$ intersects $K$ at $x$, $y\notin K'$, and $y\notin K$. It follows from the claw-freeness of $G$ that $x\neq u$. Then $d(u)+d(v)+d(x)+d(y)\geq 2|V(G)|$. It follows that $d(v)+d(x)\geq |V(G)|$ or $d(u)+d(y)\geq |V(G)|$. Then $vx\in E(G)$ or $uy\in E(G)$ by Lemma~\ref{LeComplation}, contradicting the choice of $K''$.
Thus $K\cup K'$ contains all $a$-heavy pairs.
\end{proof}

We close this section with an easy observation concerning the existence of (possibly degenerated)
generalized claws and nets. We start with the definitions. Let $x$ be a vertex, let $z_1,z_2,z_3$ be three distinct vertices (with possibly one of them being equal to $x$), and let $Q_1$, $Q_2$ and $Q_3$ be three paths with origin $x$ and termini $z_1$, $z_2$ and $z_3$, respectively, such that they only have the vertex $x$ in common. We call the union of $Q_1$, $Q_2$ and $Q_3$ a \emph{generalized claw connecting $z_1,z_2,z_3$}. If all the three paths $Q_1$, $Q_2$, $Q_3$ are non-trivial, then this generalized claw is called \emph{non-degenerate}. Similarly, let $T=x_1x_2x_3x_1$ be a triangle, let $z_1,z_2,z_3$ be three distinct vertices (with possibly $x_i=z_i$), and let
$Q_i$,
$1\leq i\leq 3$, be three disjoint paths with origin $x_i$ and terminus $z_i$. We call the union of $T$, $Q_1$, $Q_2$ and $Q_3$ a \emph{generalized net connecting $z_1,z_2,z_3$}. If all the three paths $Q_1$, $Q_2$, $Q_3$ are non-trivial, then this generalized net is called \emph{non-degenerate}. Note that a non-degenerate generalized claw (net) contains a claw (net).

\begin{lemma}\label{LeGeneralized}
Let $G$ be a connected graph, and let $z_1,z_2,z_3$ be three distinct vertices of $G$. Then $G$ has an induced generalized claw or net connecting $z_1,z_2,z_3$.
\end{lemma}

\begin{proof}
Let $P$ be a shortest path between $z_1$ and $z_2$. If $P$ passes through $z_3$, then $P$ is a generalized claw connecting $z_1,z_2,z_3$. Now we assume that $z_3\notin V(P)$. Let $P'$ be a shortest path from $z_3$ to $P$. Let $x$ be the end-vertex of $P'$ on $P$. Let $x_3$ be the neighbor of $x$ on $P'$. Note that every vertex in $V(P)$ and every vertex in $V(P')\backslash\{x,x_3\}$ are
nonadjacent.

If $x_3$ has only one neighbor $x$ in $P$, then $P\cup P'$ is an induced generalized claw connecting $z_1,z_2,z_3$.

Now we assume that $d_P(x_3)\geq 2$. Let $x_1$ and $x_2$ be the first and last neighbor of $x_3$ in $P$ (along the order of $P$ from $z_1$ to $z_2$). We claim that $d_P(x_1,x_2)\leq 2$. Otherwise $P[z_1,x_1]x_1x_3x_2P[x_2,z_2]$ is a path between $z_1$ and $z_2$ which is shorter than $P$, a contradiction.

If $d_P(x_1,x_2)=1$, i.e., $x_1x_2\in E(P)$, then $G[V(P)\cup V(P')]$ is an induced generalized net connecting $z_1,z_2,z_3$. If $d_P(x_1,x_2)=2$, then let $x'$ be the common neighbor of $x_1$ and $x_2$ in $P$. Then $G[(V(P)\cup V(P'))\backslash\{x'\}]$ is an induced generalized claw connecting $z_1,z_2,z_3$.
\end{proof}

\section{Heavy subgraph pairs}\label{heavy}

In this section, we introduce the concepts of $p$-heavy and $q$-heavy nets. We will show that every induced net of the $c$-closure of a $\{K_{1,3}, N\}$-$o$-heavy graph is $p$-heavy, and that every induced net of the $c$-closure of a $\{K_{1,3}, W\}$-$o$-heavy graph is $p$-heavy or $q$-heavy.

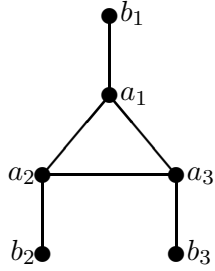
\begin{figure}[h]
  \centering
\begin{picture}(80,110)
\thicklines

\multiput(20,10)(50,0){2}{\put(0,0){\line(0,1){30}}
\multiput(0,0)(0,30){2}{\circle*{6}}}
\multiput(45,70)(0,30){2}{\put(0,0){\circle*{6}}}
\put(45,70){\line(0,1){30}} \put(20,40){\line(1,0){50}}
\put(20,40){\line(5,6){25}} \put(70,40){\line(-5,6){25}}

\put(49,68){$a_1$} \put(49,98){$b_1$} \put(74,38){$a_3$}
\put(74,8){$b_3$} \put(7,38){$a_2$} \put(8,8){$b_2$}

\end{picture}\caption{\small The graph $N$ with labels.}\label{N}
\end{figure}

In the next definitions, we refer to the vertices of an induced net $N$ according to the labels in Figure~\ref{N}. Recall that $N$ is called $o$-heavy if $N$ contains an $o$-heavy pair. If the triangle of $N$ contains an $a$-heavy pair, then $N$ is called \emph{$p$-heavy}.

We say that $N$ is \emph{q-heavy} if there exists an $i\in\{1,2,3\}$ such that\\
(i) $a_i,b_i$ are heavy in $G$;\\
(ii) $d(a_j)=3,d(b_j)=2$ for $j\neq i$; and\\
(iii) the neighbor of $b_j$ other than $a_j$ is heavy for $j\neq i$.

Recall that $G$ is $N$-$o$-heavy if every induced net of $G$ is $o$-heavy. We say that $G$ is \emph{N-$p$-heavy} if every induced net of $G$ is $p$-heavy; we also say that $G$ is \emph{$N$-$op$-heavy} (\emph{N-$pq$-heavy}) if every induced net of $G$ is either $o$-heavy or $p$-heavy (either $p$-heavy or $q$-heavy).

As we pointed out before, the $c$-closure of a $\{K_{1,3},N\}$-$o$-heavy graph is not necessarily $N$-free. However, we can prove the following weaker result.

\begin{theorem}\label{ThCN}
Let $G$ be a $\{K_{1,3},N\}$-$o$-heavy graph. Then $\clc(G)$ is claw-free and $N$-$p$-heavy.
\end{theorem}

Moreover, we can prove the following more general result.

\begin{theorem}\label{ThCPCZN}
Let $G$ be a 2-connected $\{K_{1,3},S\}$-$o$-heavy graph, where $S\in\{P_4,P_5,C_3,Z_1,Z_2,B,N\}$. Then $\clc(G)$ is claw-free and $N$-$p$-heavy.
\end{theorem}

The $c$-closure of a $\{K_{1,3},W\}$-$o$-heavy graph $G$ is not necessarily $N$-$p$-heavy. However, we can prove that for such a graph $G$ every net of $\clc(G)$ is either $p$-heavy or $q$-heavy.

\begin{theorem}\label{ThCW}
Let $G$ be a 2-connected $\{K_{1,3},W\}$-$o$-heavy graph. Then $\clc(G)$ is claw-free and $N$-$pq$-heavy.
\end{theorem}

Combining Theorem~\ref{ThCPCZN} with Theorem~\ref{ThCW}, we conclude that the following common property holds.

\begin{theorem}\label{ThCPCZNW}
Let $G$ be a 2-connected $\{K_{1,3},S\}$-$o$-heavy graph, where $S\in\{P_4,P_5,C_3,Z_1,Z_2,B,N,W\}$. Then $\clc(G)$ is claw-free and $N$-$pq$-heavy.
\end{theorem}

These results form the basis and motivation for characterizing the structure of all 2-connected, $c$-closed, claw-free and $N$-$pq$-heavy graphs in Section~\ref{c-closed}.
The next section contains the proofs of the above results.

\section{Proofs of Theorems~\ref{ThCN}, \ref{ThCPCZN} and \ref{ThCW}}\label{proofs}

Before presenting the proofs, we recall the following lemma from~\cite{LiNing}. Theorem~\ref{ThCN} is almost a direct consequence of this result.

\begin{lemma}[Li and Ning \cite{LiNing}]\label{LeopHeavy}
Let $G$ be a claw-$o$-heavy and $N$-$op$-heavy graph. Then $\clc(G)$ is $N$-$op$-heavy.
\end{lemma}

We first prove a similar result concerning $N$-$pq$-heavy graphs.

\begin{lemma}\label{LepqHeavy}
Let $G$ be a claw-$o$-heavy and $N$-$pq$-heavy graph. Then $\clc(G)$ is
$N$-$pq$-heavy.
\end{lemma}

\begin{proof}
Let $M$ be an induced net in $G'=\clc(G)$. We label the vertices of $M$ as in Figure~\ref{N}. We will prove that $M$ is either $p$-heavy or
$q$-heavy. Let $n=|V(G)|$.

Let $R$ be the region of $G$ containing $\{a_1,a_2,a_3\}$, and let
$R_i$ be the region of $G$ containing $\{a_i,b_i\}$. Note that $a_i$
is a frontier vertex in $R\cap R_i$. If $R_i$ has an interior
vertex, then let $b'_i$ be an interior neighbor of $a_i$ in $R_i$
(see Lemma~\ref{LeRegion}); otherwise ($R_i$ is complete and) let
$b'_i=b_i$. We claim that $b'_ib'_j\notin E(G')$ for $1\leq i<j\leq
3$. Suppose to the contrary that $b'_1b'_2\in E(G')$. Then either
$b'_1\neq b_1$ or $b'_2\neq b_2$. We assume without loss of
generality that $b'_1\neq b_1$, and then $b'_1$ is an interior vertex
of $G$. Then both $b'_2,a_1$ are associated with $b'_1$, and then
$b'_2,a_1$ themselves are associated. Since $b'_2$ is associated
with $a_2$ as well, by Lemma~\ref{LeRegion}, $b'_2\in R$, a
contradiction. Thus as we claimed, $b'_ib'_j\notin E(G')$ for $1\leq
i<j\leq 3$.

Let $I_R$ be the set of interior vertices in $R$. By Lemma~\ref{LeRegion},
there is a path between $a_i$ and $a_j$ with all
internal vertices in $I_R$ ($1\leq i<j\leq 3$). This implies that
$a_1,a_2,a_3$ are connected in the graph
$G[I_R\cup\{a_1,a_2,a_3\}]$. By Lemma~\ref{LeGeneralized}, let $H$
be an induced generalized claw or net in $G[I_R\cup\{a_1,a_2,a_3\}]$
connecting $a_1,a_2,a_3$. Let $H'=G[V(H)\cup\{b'_1,b'_2,b'_3\}]$.
Then $H'$ is a non-degenerate generalized claw or net. Thus there is
an induced claw or net $H''$ of $G$ contained in $H'$.

Since $G$ is claw-$o$-heavy and $N$-$pq$-heavy, there are two vertices
contained in $H'$ with degree sum at least $n$. Let $u,v$ be such
two vertices. Then $d_{G'}(u)+d_{G'}(v)\geq d(u)+d(v)\geq n$. If
$u,v\in V(H)$, then note that $d_{G'}(w)\leq d_{G'}(a_i)$ for
$w\in I_R$ and $a_i\in\{a_1,a_2,a_3\}$. This implies that there are
two vertices in $\{a_1,a_2,a_3\}$ with degree sum at least $n$ in
$G'$, i.e., $M$ is $p$-heavy.

Now we assume that one of $u,v$ is not in $V(H)$. We assume without
loss of generality that $u=b'_1$. By Lemma~\ref{LeClosure},
$v=a_1$. This implies that $H''$ is a $q$-heavy net containing
$a_1,b'_1$, and $a_1,b'_1$ are two heavy vertices of $G$. By the
definition of $q$-heavy nets, $R$ has only three vertices
$a_1,a_2,a_3$, $d(a_2)=d(a_3)=3$ and $d(b'_2)=d(b'_3)=2$. This implies
that $V(R_2)=\{a_2,b_2\}$ and $V(R_3)=\{a_3,b_3\}$. Let $c_2$ and
$c_3$ be the neighbors of $b_2$ and $b_3$ other than $a_2$ and
$a_3$, respectively. By Lemma~\ref{LeRegion} and the fact that
$d(b_2')=d(b_3')=2$, {the two regions containing $\{b_2,c_2\}$,
$\{b_3,c_3\}$, respectively, have only two vertices.}
Note that $c_2$ and $c_3$ are
heavy in $G$. Thus $M$ is $q$-heavy in $\clc(G)$.
\end{proof}

\subsection*{Proof of Theorem \ref{ThCN}.}

Let $G$ be a $\{K_{1,3},N\}$-$o$-heavy graph. By Theorem~\ref{ThCa}
and Lemma~\ref{LeopHeavy}, $\clc(G)$ is claw-free and $N$-$op$-heavy.
Note that $\clc(G)$ has no $o$-heavy pairs. This implies that every
induced net of $\clc(G)$ is not $o$-heavy, and thus, is $p$-heavy. Hence
$\clc(G)$ is claw-free and $N$-$p$-heavy.
\qed

\subsection*{Proof of Theorem \ref{ThCPCZN}.}

The proof of this theorem is based on Theorem~\ref{ThCW}, which will
be proved later.

Note that a $P_4$-$o$-heavy ($C_3$-$o$-heavy, $Z_1$-$o$-heavy or
$B$-$o$-heavy) graph is also $N$-$o$-heavy. Hence, the cases $S=P_4,C_3,Z_1,B$
can be deduced from Theorem~\ref{ThCN} immediately. Now we assume that
$S=P_5$ or $Z_2$.

Let $G$ be a $\{K_{1,3},P_5\}$-$o$-heavy or $\{K_{1,3},Z_2\}$-$o$-heavy
graph and let $G'=\clc(G)$. By Theorems~\ref{ThCa} and \ref{ThCW},
$G'$ is claw-free and $N$-$pq$-heavy. Let $M$ be an induced net of
$G'$. We will prove that $M$ is $p$-heavy. Suppose to the contrary that $M$ is $q$-heavy. We label the vertices of $M$ as in Figure~\ref{N} such that $a_1,b_1$ are heavy in $G'$, $d_{G'}(a_2)=d_{G'}(a_3)=3$ and $d_{G'}(b_2)=d_{G'}(b_3)=2$. Let $c_2$ and $c_3$ be the neighbors of $b_2$ and $b_3$ other than $a_2$ and $a_3$, respectively. Then $c_2,c_3$ are heavy in $G'$. Clearly $d(a_2)=d(a_3)=3$ and $d(b_2)=d(b_3)=2$. Recall that each two heavy vertices of $G'$ are adjacent in $G'$. Thus $a_1,b_1,c_2,c_3$ are contained in a common region $R$.

Note that $c_2b_2a_2a_3b_3$ is an induced copy of $P_5$ in $G$. But any two nonadjacent vertices of this $P_5$ are dissociated. This implies that $G$ is not $P_5$-$o$-heavy. So we conclude that $G$ is $Z_2$-$o$-heavy.

Note that $b_1\neq c_2,c_3$ (for otherwise $b_1b_2$ or $b_1b_3\in E(G')$). Since $b_1$ is heavy in $G'$, and $b_1$ is nonadjacent to $b_1$ itself and $a_2,a_3,b_2,b_3$ in $G'$, this implies that $|V(G')|\geq 10$. Next we prove two claims.

\begin{claim}\label{ClTriangle}
Neither $a_1$ nor $c_2$ is contained in a triangle of $G$ other than $a_1a_2a_3a_1$ and $a_1c_2c_3a_1$.
\end{claim}

\begin{proof}
Suppose that $a_1$ is contained in a triangle $T$ in $R$ with $T\neq a_1c_2c_3a_1$. We assume without loss of generality that $c_2\notin V(T)$. Then the subgraph induced by $V(T)\cup\{a_2,b_2\}$ is a $Z_2$, and each two nonadjacent vertices of the $Z_2$ are dissociated, a contradiction.

Suppose that $c_2$ is contained in a triangle $T\neq a_1c_2c_3a_1$. By the above analysis, $a_1\notin V(T)$. Then the subgraph induced by $V(T)\cup\{b_2,a_2\}$ is a $Z_2$, and each two nonadjacent vertices of the $Z_2$ are dissociated, a contradiction.
\end{proof}

\begin{claim}\label{ClHeavy}
Both $a_1$ and $c_2$ are not heavy in $G$.
\end{claim}

\begin{proof}
If $c_2$ is heavy in $G$, then recalling that $|V(G')|\geq 10$, there are
two vertices $c'_2,c''_2\in N(c_2)\backslash\{a_1,b_2,c_3\}$. By
Claim~\ref{ClTriangle}, $c'_2c''_2\notin E(G)$, and the subgraph of
$G$ induced by $\{c_2,b_2,c'_2,c''_2\}$ is a claw. Since $d(b_2)=2$,
either $c'_2$ or $c''_2$ is heavy in $G$. We assume without loss of
generality that $c'_2$ is heavy. Since $d(c_2)+d(c'_2)\geq n$ and
$a_2\notin N(c_2)\cup N(c'_2)$, $c_2$ and $c'_2$ have a common
neighbor $c'''_2$. Then $c_2c'_2c'''_2c_2$ is a triangle other than
$a_1c_2c_3a_1$, a contradiction. Hence, we conclude that $c_2$ is not
heavy in $G$.

If $a_1$ is heavy in $G$, then there are two neighbors $a'_1,a''_1$
of $a_1$ in $R$ other than $c_3$. By Claim~\ref{ClTriangle},
$a'_1a''_1\notin E(G)$ and one of $\{a'_1,a''_1\}$ is heavy in $G$.
We assume without loss of generality that $a'_1$ is heavy in $G$.
Then $a'_1\neq c_2$. Note that $d(a_1)+d(a'_1)\geq n$ and $b_2\notin
N(a_1)\cup N(a'_1)$. Thus $a_1$ and $a'_1$ have a common neighbor
$a'''_1$. Hence, $a_1a'_1a'''_1a_1$ is a triangle other than
$a_1c_2c_3a_1$, a contradiction. We conclude that $a_1$ is not
heavy in $G$.
\end{proof}

If $a_1c_2\notin E(G)$, then the subgraph induced by $\{a_1,a_3,a_2,b_2,c_2\}$ is a $Z_2$. This implies that $\{a_1,c_2\}$ is an $o$-heavy pair of $G$, and either $a_1$ or $c_2$ is heavy in $G$, contradicting Claim~\ref{ClHeavy}. Thus we conclude
that $a_1c_2\in E(G)$.

If $R$ has no interior vertices, then $R$ is complete and $a_1b_1c_2a_1$ is a triangle other than $a_1c_2c_3a_1$, a contradiction. So we conclude that $R$ has an interior vertex. By Lemma~\ref{LeRegion}, let $a'_1$ and $c'_2$ be interior neighbors of $a_1$ and $c_2$ in $R$, respectively. Then $a_1c'_2,a'_1c_2\notin E(G)$. Thus the subgraphs of $G$ induced by $\{a_1,a_2,c_2,a'_1\}$ and $\{c_2,a_1,b_2,c'_2\}$ are claws. Hence, $d(a'_1)+d(c_2)\geq n$
and $d(a_1)+d(c'_2)\geq n$. This implies that either $d(a_1)+d(a'_1)\geq n$ or $d(c_2)+d(c'_2)\geq n$. Note that $b_3\notin N(a_1)\cup N(a'_1)\cup N(c_2)\cup N(c'_2)$. Thus either $a_1,a'_1$ or $c_2,c'_2$ have a common neighbor. This implies that either $a_1$ or $c_2$ is contained in a triangle other than $a_1c_2c_3a_1$, a contradiction.

This completes the proof of Theorem~\ref{ThCPCZN}. \qed

\subsection*{Proof of Theorem \ref{ThCW}.}

Let $M$ be an induced net of $G'=\clc(G)$. We suppose to the contrary that $M$ is neither $p$-heavy nor $q$-heavy. We label the vertices of $M$ as in Figure~\ref{N}. Let $n=|V(G)|$.

Let $R$ be the region of $G$ containing $\{a_1,a_2,a_3\}$, and let $R_i$ be the region of $G$ containing $\{a_i,b_i\}$. Note that $a_i$ is a frontier vertex in $R_i\cap R$. If $R_i$ has an interior vertex, then let $b'_i$ be an interior neighbor of $a_i$ in $R_i$; otherwise let $b'_i=b_i$. By Lemma~\ref{LeRegion}, $b'_ib'_j\notin
E(G')$ for $1\leq i<j\leq 3$. Next we prove four claims.

{\setcounter{claim}{0}}
\begin{claim}\label{ClR}
$a_1a_2,a_1a_3,a_2a_3\in E(G)$.
\end{claim}

\begin{proof}
Let $I_R$ be the set of interior vertices in $R$. By Lemma~\ref{LeRegion},
$a_1,a_2,a_3$ are connected in $G[I_R\cup\{a_1,a_2,a_3\}]$. By Lemma~\ref{LeGeneralized}, there is a generalized claw or net $H$ in $G[I_R\cup\{a_1,a_2,a_3\}]$ connecting $a_1,a_2,a_3$. Thus $H'=G[V(H)\cup\{a'_1,a'_2,a'_3\}]$ is a non-degenerate generalized claw or net.

If one of $a_1a_2,a_1a_3,a_2a_3$ is not in $E(G)$, then $H'$ contains an induced claw or wounded of $G$, implying that $H'$ has two nonadjacent vertices with degree sum at least $n$ in $G$. Note that any two dissociated vertices have degree sum in $G$ less than
$n$. Hence, there are two vertices in $I_R\cup\{a_1,a_2,a_3\}$ with degree sum at least $n$. Also note that $d_{G'}(a_i)\geq d_{G'}(v)\geq d(v)$ for every $v\in I_R$. This implies that there are two vertices in $\{a_1,a_2,a_3\}$ with degree sum at least $n$ in $G'$. Hence, $M$ is $p$-heavy, a contradiction.
\end{proof}

We suppose without loss of generality that $a_1$ has the largest degree in $G$ among $\{a_1,a_2,a_3\}$. Since $G$ is 2-connected, every region has at least two frontier vertices. For $i=2,3$, let $c_i$ be a frontier vertex of $R_i$ other than $a_i$, let $c'_i$ be a neighbor of $c_i$ outside $R_i$, and let $R'_i$ be the region of $G$ containing $\{c_i,c'_i\}$.

By Lemma~\ref{LeRegion}, $a_1$ is dissociated with all vertices in
{$R_2\setminus \{a_2\}$}, in particular
$a_1c_2\notin E(G')$. By a similar analysis, $a_1c_3,a_2c_1,a_2c_3,a_3c_1,a_3c_2\notin E(G')$.

\begin{claim}\label{ClR23}
For $i=2,3$, (1) $a_ic_i\in E(G)$; (2) $a_1c'_i\in E(G')$; (3)
$V(R'_i)=\{c_i,c'_i\}$; and (4) $V(R_i)=\{a_i,c_i\}$.
\end{claim}

\begin{proof}
By symmetry, we assume $i=2$.

(1) Suppose that $a_2c_2\notin E(G)$. Then let $P$ be a path from $a_2$ to $c_2$ all internal vertices of which are interior vertices in $R_2$. Let $a_2,a'_2,a''_2$ be the first three vertices of $P$. If $b'_1c_2\notin E(G)$, then the subgraph of $G$ induced by
$\{a_3,a_1,b'_1,a_2,a'_2,a''_2\}$ is a wounded; if $b'_1c_2\in E(G)$, then $b'_3c_2\notin E(G)$, and the subgraph of $G$ induced by $\{a_1,a_3,b'_3,a_2,a'_2,a''_2\}$ is a wounded. In any case, we have $d(a_2)+d(a''_2)\geq n$ (since any other nonadjacent pairs contained
in the wounded are dissociated). Since $d_{G'}(a_1)\geq d_{G'}(a_2)\geq d(a_2)$, we have $d_{G'}(a_1)+d_{G'}(a''_2)\geq n$, and $\{a_1,a''_2\}$ is an $o$-heavy pair of $G'$, a contradiction.

(2) Suppose to the contrary that $a_1c'_2\notin E(G')$ (in particular $a_1c'_2\notin E(G)$).

First we assume that $a_3c'_2\in E(G)$. Then $c'_2\in V(R_3)$. If
$b'_3c'_2\notin E(G)$, then ($R_3$ is not complete, $b'_3$ is an
interior vertex and) the subgraph of $G$ induced by
$\{a_3,a_2,b'_3,c'_2\}$ is a claw. This implies that
$d(b'_3)+d(c'_2)\geq n$. Since $d_{G'}(a_1)\geq d_{G'}(a_3)\geq
d_{G'}(b'_3)\geq d(b'_3)$, $d_{G'}(a_1)+d_{G'}(c'_2)\geq n$,
implying $\{a_1,c'_2\}$ is an $o$-heavy pair of $G'$, a contradiction.
Thus we conclude that $b'_3c'_2\in E(G)$, and the subgraph of $G$
induced by $\{b'_3,c'_2,c_2,a_3,a_1,b'_1\}$ is a wounded. But any
two nonadjacent vertices in the wounded are dissociated, a
contradiction. So we have that $a_3c'_2\notin E(G)$.

Now suppose that $b'_1c'_2\in E(G)$. If $b'_1$ and $c_2$ are
associated, or $b'_3$ and $c_2$ are associated, then $c_2\neq b'_2$
and the subgraph induced by $\{a_3,a_2,b'_2,a_1,b'_1,c'_2\}$ is a
wounded. But any two vertices of the wounded are dissociated, a
contradiction. This implies that $b'_1c_2,b'_3c_2\notin E(G')$. If
$b'_3c'_2\in E(G')$, then the subgraph of $G'$ induced by
$\{c'_2,b'_1,b'_3,c_2\}$ is a claw, a contradiction. This implies
that $b'_3$ and $c'_2$ are dissociated. Now the subgraph induced by
$\{a_1,a_3,b'_3,a_2,c_2,c'_2\}$ is a wounded, and any two
nonadjacent vertices of the wounded are dissociated, a
contradiction. So we have that $b'_1c'_2\notin E(G)$.

Now the subgraph induced by $\{a_3,a_1,b'_1,a_2,c_2,c'_2\}$ is a
wounded, and any two nonadjacent vertices of the wounded are
dissociated, a contradiction.

(3) If $R_2$ has a third vertex $c''_2$, then similarly as in (2), we
can see that $a_1c''_2\in E(G')$. By Lemma~\ref{LeRegion}, $a_1\in
V(R'_2)$ and $c_2\in V(R)$, a contradiction.

(4) Suppose that $R_2$ has a third vertex. If $R_2$ has an interior
vertex, then let $a'_2$ be an interior neighbor of $a_2$ in $R_2$;
otherwise let $a'_2$ be a arbitrary vertex of
$V(R_2)\backslash\{a_2,c_2\}$. We claim that $a'_2c_2\in E(G)$;
otherwise $a'_2$ is an interior vertex and $\{a'_2,c_2\}$ is an
$o$-heavy pair of $G$. Since $d_{G'}(a_1)\geq d_{G'}(a_2)\geq
d_{G'}(a'_2)\geq d(a'_2)$, $\{a_1,c_2\}$ is an $o$-heavy pair of $G'$,
a contradiction. By (2), $c'_2\in V(R_1)$. Clearly $a'_2\notin
V(R_1)$. Note that $|V(R_1)|\geq|V(R_2)|\geq 3$. Let $c''_2$ be a
neighbor of $c'_2$ in $V(R_1)\backslash\{a_1\}$. Then the subgraph
of $G$ induced by $\{a'_2,a_2,a_3,c_2,c'_2,c''_2\}$ is a wounded.
But each two nonadjacent vertices of the wounded are dissociated, a
contradiction.
\end{proof}

\begin{claim}\label{ClR'}
$V(R)=\{a_1,a_2,a_3\}$.
\end{claim}

\begin{proof}
Suppose not. If $R$ has an interior vertex, then let $a$ be an
interior neighbor of $a_1$; otherwise let $a$ be an arbitrary vertex
in $V(R)\backslash\{a_1,a_2,a_3\}$. We claim that $aa_2,aa_3\in
E(G)$. Otherwise $a$ is interior and by Claim~\ref{ClR}, the
subgraph of $G$ induced by $\{a_1,a_2,a,b_1\}$ or
$\{a_1,a_3,a,b_1\}$ is a claw. This implies that $\{a,a_2\}$ or
$\{a,a_3\}$ is an $o$-heavy pair. Note that $d_{G'}(a_1)\geq
d_{G'}(a)\geq d(a)$. This implies that the degree sum of $a_1,a_2$,
or of $a_1,a_3$ in $G'$ is at least $n$, and $M$ is $p$-heavy, a
contradiction.

Now the subgraph induced by $\{a,a_3,c_3,a_2,c_2,c'_2\}$ is a
wounded. But each two nonadjacent vertices of the wounded are
dissociated, a contradiction.
\end{proof}

By Claim \ref{ClR}, \ref{ClR23} and \ref{ClR'}, we can see that
$b_2=c_2$, $b_3=c_3$. Also note that $d_{G'}(a_2)=d_{G'}(a_3)=3$ and
$d_{G'}(b_2)=d_{G'}(b_3)=2$.

\begin{claim}\label{ClHeavy'}
$a_1,b_1,c'_2,c'_3$ are heavy.
\end{claim}

\begin{proof}
We first suppose that $a_1c'_2,a_1c'_3\notin E(G)$. Then
the
subgraph of $G$ induced by $\{a_1,a_3,b_3,a_2,b_2,c'_2\}$ is a
wounded. This implies that $d(a_1)+d(c'_2)\geq n$. Note that
$d_{G'}(a_1)\geq d(a_1)+2$ and $d_{G'}(c'_2)\geq d(c'_2)+1$. Thus
$d_{G'}(a_1)+d_{G'}(c'_2)\geq n+3$. Note that
$d_{G'}(a_1)=d_{G'}(c'_2)+1$, $d_{G'}(c'_3)=d_{G'}(c'_2)$ and
$d_{G'}(b_1)\geq d_{G'}(c'_2)-1$. This implies that $d_{G'}(a_1)\geq
n/2+2$, $d_{G'}(c'_2)=d_{G'}(c'_3)\geq n/2+1$ and $d_{G'}(b_1)\geq
n/2$.

Now we assume without loss of generality that $a_1c'_2\in E(G)$. If
$b'_1c'_2\notin E(G)$, then $b'_1$ is interior and $\{b'_1,c'_2\}$
is an $o$-heavy pair. Note that $d_{G'}(b'_1)\geq d(b'_1)+1$ and
$d_{G'}(c'_2)\geq d(c'_2)+1$. Thus $d_{G'}(b'_1)+d_{G'}(c'_2)\geq
n+2$. Note that $d_{G'}(b'_1)=d_{G'}(c'_2)-1$. We have
$d_{G'}(b'_1)\geq n/2$. Thus $a_1,b_1,c'_2,c'_3$ are heavy. Now we
assume that $b'_1c'_2\in E(G)$. Then the subgraph induced by
$\{b'_1,c'_2,b_2,a_1,a_3,b_3\}$ is a wounded. But each two
nonadjacent vertices of the wounded are dissociated, a
contradiction.
\end{proof}

By Claims \ref{ClR23}, \ref{ClR'} and \ref{ClHeavy'}, $M$ is
$q$-heavy, a contradiction.

This completes the proof of Theorem~\ref{ThCW}. \qed

\section{Structure of $c$-closed graphs}\label{c-closed}

In this section, we are going to characterize the structure of all 2-connected, $c$-closed, claw-free and $N$-$pq$-heavy graphs. For this we need the following classes of graphs, as well as the classes $\mathcal{C}_1^N$ and $\mathcal{C}_2^N$ we defined before in Section~\ref{intro} (see Figures~\ref{C_1^N} and \ref{C_2^N}).

\begin{construction}
Denote by $\mathcal{C}_3^{NQ}$ the class of graphs obtained by the following construction: (i) let $K$ be a complete graph with $|V(K)|\geq 4$ and let $b_2a_2a_3b_3$ be a path of length three outside $K$; (ii) choose three distinct vertices $a_1,c_2,c_3$ in $K$; and (iii) add edges $c_2b_2$, $a_1a_2$, $a_1a_3$, and $c_3b_3$, as shown in Figure~\ref{G_3}. It is clear that every graph in $\mathcal{C}_3^{NQ}$ is hamiltonian.
\end{construction}

\begin{figure}[htbp]\centering
\begin{tikzpicture}[scale=0.5]

 \draw[fill=lightgray](0,0)circle(3);
 \path(2.3,1.5) coordinate  (a);\draw [fill=black] (a) circle (0.15cm);
  \path(2.3,-1.5) coordinate  (b);\draw [fill=black] (b) circle (0.15cm);
  \path(2,0) coordinate  (c);\draw [fill=black] (c) circle (0.15cm);
  \path(3.8,1.5) coordinate  (d);\draw [fill=black] (d) circle (0.15cm);
  \path(3.8,-1.5) coordinate  (e);\draw [fill=black] (e) circle (0.15cm);
  \path(5.4,1.5) coordinate  (f);\draw [fill=black] (f) circle (0.15cm);
    \path(5.4,-1.5) coordinate  (g);\draw [fill=black] (g) circle (0.15cm);
\draw [line width=0.8] (f)--(g);
\draw [line width=0.8] (f)--(a);
\draw [line width=0.8] (b)--(g);
\draw [line width=0.8] (f)--(c);
\draw [line width=0.8] (c)--(g);
\node [above] at (2.3,1.5){$c_2$};
\node [below] at (2.3,-1.5){$c_3$};
\node [left] at (2,0){$a_1$};
\node [above] at (3.8,1.5){$b_2$};
\node [below] at (3.8,-1.5){$b_3$};
\node [above] at (5.4,1.5){$a_2$};
\node [below] at (5.4,-1.5){$a_3$};
\end{tikzpicture}
\caption{The class of graphs $\mathcal{C}_3^{NQ}$.}\label{G_3}
\end{figure}
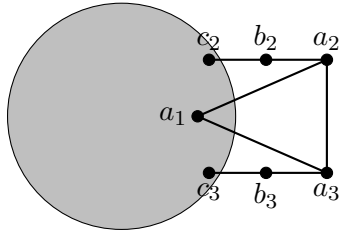

We use the graphs from the classes $\mathcal{C}_1^N$, $\mathcal{C}_2^N$ and $\mathcal{C}_3^{NQ}$ as ingredients for  defining the next four classes of more complicated graphs. All these graphs have in common that they consist of a clique or two cliques intersecting in one vertex, and subgraphs from $\mathcal{C}_1^N$, $\mathcal{C}_2^N$ or $\mathcal{C}_3^{NQ}$ that intersect these cliques in a particular way.

\begin{construction}
Denote by $\mathcal{C}_1^{NP}$ the class of graphs, each member $G$ of which has the following properties.
\begin{itemize}
\item [(i)] $G$ has a maximal clique $K$ containing all heavy vertices;
\item [(ii)] for each component $H$ of $G-K$, one of the following holds:
$$(a)~G[V(H)\cup K]\in \mathcal{C}_1^N,~~(b)~G[V(H)\cup K]\in \mathcal{C}_2^N,$$
and moreover, $G-K$ has at least two components; and if $G-K$ has exactly two components, then at least one of them satisfies ($b$); and
\item [(iii)] letting $u_1,u_2,u_3$ be three distinct frontier vertices of $K$ and letting $v_i$, $i=1,2,3$, be a neighbor of $u_i$ outside $K$, if $\{v_1,v_2,v_3\}$ is independent in $G$, then $\{u_1,u_2,u_3\}$ contains an $a$-heavy pair.
\end{itemize}
\end{construction}

\begin{construction}\label{CoC_2^NP}
Denote by $\mathcal{C}_2^{NP}$ the class of graphs, each member $G$ of which has the following properties.
\begin{itemize}
\item [(i)]$G$ has a maximal clique $K$ containing all heavy vertices;
\item [(ii)] there is a maximal clique $K'$ such that $K$ intersects $K'$ at one vertex, say $u_0$;
\item [(iii)] each component $H$ of $G-(K\cup K')$  satisfies at least one of the following:
\begin{gather*}
    (a)~N(H)\subset K~\text{and}~G[V(H)\cup K]\in \mathcal{C}_1^N,~~(b)~N(H)\subset K~\text{and}~G[V(H)\cup K]\in \mathcal{C}_2^N,\\~~(c)~N(H)\subset K'~\text{and}~G[V(H)\cup K']\in \mathcal{C}_1^N,~~(d)~N(H)\cap K'\neq \emptyset~\text{and}~N(H)\cap K\neq\emptyset~\\ \text{and}~G[V(H)\cup K\cup K']\in \mathcal{C}_2^N,
\end{gather*}
and moreover, there is at least one component satisfying (d), and there are exactly two components satisfying (c) or (d);
\item [(iv)] letting $u_1,u_2$ be two distinct frontier vertices other than $u_0$ of {$G[K']$}, and letting $v_i$, $i=0,1,2$, be a neighbor of $u_i$ outside $G[K']$, if $\{v_0,v_1,v_2\}$ is independent in $G$, then $\{u_0,u_1,u_2\}$ contains an $a$-heavy pair; and
\item [(v)] letting $x_1,x_2$ be two distinct frontier vertices other than $u_0$ of $K$, and letting $y_0$ be a neighbor of $u_0$ outside $G[K]$, and letting $y_i$, $i=1,2$, be a neighbor of $x_i$ outside $G[K]$, if $\{u_0,x_1,x_2\}$ is independent in $G$, then $\{y_0,y_1,y_2\}$ contains an $a$-heavy pair.
\end{itemize}
\end{construction}

\begin{construction}\label{CoC_1^NPQ}
Denote by $\mathcal{C}_1^{NPQ}$ the class of graphs, each member $G$ of which has the following properties.
\begin{itemize}
\item [(i) ] $G$ has a clique $K$ with $|K|\geq\frac{|V(G)|}{2}$; and
\item [(ii)] for each component $H$ of $G-K$, one of the following holds:
$$(a)~G[V(H)\cup K]\in \mathcal{C}_1^N,~~(b)~G[V(H)\cup K]\in \mathcal{C}_2^N,~~(c)~G[V(H)\cup K]\in \mathcal{C}_3^{NQ},$$
and moreover, there is at least one component $H$ satisfying (c).
\end{itemize}
\end{construction}

\begin{construction}\label{CoC_2^NPQ}
Denote by $\mathcal{C}_2^{NPQ}$ the class of graphs, each member $G$ of which has the following properties.
\begin{itemize}
\item [(i) ] $G$ has a clique $K$ with $|K|\geq\frac{|V(G)|}{2}$;
\item [(ii) ] there is a maximal clique $K'$ such that $K$ intersects $K'$ at one vertex, say $u_0$;
\item [(iii)] each component $H$ of $G-(K\cup K')$ satisfies one of the following:
\begin{gather*}
(a)~N(H)\subset K~\text{and}~G[V(H)\cup K]\in \mathcal{C}_1^N,~~(b)~N(H)\subset K~\text{and}~G[V(H)\cup K]\in \mathcal{C}_2^N,\\~~(c)~N(H)\subset K~\text{and}~G[V(H)\cup K]\in \mathcal{C}_3^{NQ},
~~(d)~N(H)\subset K'~\text{and}~G[V(H)\cup K']\in \mathcal{C}_1^N,\\
~~(e)~N(H)\cap K\neq\emptyset~\text{and}~N(H)\cap K'\neq\emptyset~\text{and}~G[V(H)\cup K\cup K']\in \mathcal{C}_2^N,
\end{gather*}
and moreover, there is at least one component satisfying (c), there is at least one component satisfying (e), and there are exactly two components satisfying (d) or (e); and
\item [(iv)] letting $u_1,u_2$ be two distinct frontier vertices other than $u_0$ of $G[K']$, and letting $v_i$, $i=0,1,2$, be
a
 neighbor of $u_i$ outside $G[K']$, if $\{v_0,v_1,v_2\}$ is independent in $G$, then $\{u_0,u_1,u_2\}$ contains an $a$-heavy pair.
\end{itemize}
\end{construction}

It is straightforward to verify the next observation.

\begin{observation}
Every graph in $\mathcal{C}_1^{N}\cup \mathcal{C}_2^{N}\cup \mathcal{C}_1^{NP}\cup \mathcal{C}_2^{NP}\cup \mathcal{C}_1^{NPQ}\cup \mathcal{C}_2^{NPQ}$ is hamiltonian.
\end{observation}

Motivated and inspired by Ryj\'a\v{c}ek's characterization of the structure of $r$-closed $\{K_{1,3},N\}$-free graphs (Theorem~\ref{Ryjacek'N}), we characterize the structure of  $c$-closed, claw-free and $N$-$pq$-heavy graphs. We start with the counterpart for $N$-$p$-heavy graphs.

\begin{theorem}\label{ThCP}
Let $G$ be a graph {of order $n\geq 10$}. Then $G$ is a 2-connected, $c$-closed, claw-free and $N$-$p$-heavy graph if and only if $G\in \mathcal{C}_1^N\cup\mathcal{C}_2^N\cup \mathcal{C}_1^{NP}\cup \mathcal{C}_2^{NP}$.
\end{theorem}

Our main result of this section is the following more general result.

\begin{theorem}\label{ThCPQ}
Let $G$ be a graph {of order $n\geq 10$}. Then $G$ is a 2-connected, $c$-closed, claw-free and $N$-$pq$-heavy graph if and only if $G\in \mathcal{C}_1^N\cup\mathcal{C}_2^N\cup \mathcal{C}_1^{NP}\cup \mathcal{C}_2^{NP}\cup\mathcal{C}_1^{NPQ}\cup \mathcal{C}_2^{NPQ}$.
\end{theorem}

Before proving the above theorems, we finish with some consequences of our results. By Lemma~\ref{LepqHeavy} and Theorems~\ref{ThCPQ} and \ref{ThCa}, we can extend Theorem~\ref{ThBuGoJa} in the following way.

\begin{theorem}\label{ThNpqH}
Every 2-connected claw-free and $N$-$pq$-heavy graph is hamiltonian.
\end{theorem}

Note that we also have a new proof for the `if' part of Theorem~\ref{ThLiRyWaZh} by  combining Theorems~\ref{ThCa}, \ref{ThCPCZNW} and \ref{ThNpqH}. Combining Theorem~\ref{ThCPCZN} with Theorem~\ref{ThCP}, we obtain the following corollary.

\begin{corollary}
Let $G$ be a 2-connected $\{K_{1,3},S\}$-$o$-heavy graph {of order $n\geq 10$}, where
$S\in\{P_4,P_5,C_3,Z_1,\\Z_2,B,N\}$. Then $\clc(G)\in \mathcal{C}_1^{N}\cup \mathcal{C}_2^{N}\cup \mathcal{C}_1^{NP}\cup \mathcal{C}_2^{NP}$.
\end{corollary}

Combining Theorem~\ref{ThCPCZNW} with Theorem~\ref{ThCPQ}, we obtain the following corollary.

\begin{corollary}
Let $G$ be a 2-connected $\{K_{1,3},S\}$-$o$-heavy graph {of order $n\geq 10$}, where
$S\in\{P_4,P_5,C_3,Z_1,\\Z_2,B,N,W\}$. Then $\clc(G)\in \mathcal{C}_1^{N}\cup \mathcal{C}_2^{N}\cup \mathcal{C}_1^{NP}\cup \mathcal{C}_2^{NP}\cup\mathcal{C}_1^{NPQ}\cup \mathcal{C}_2^{NPQ}$.
\end{corollary}

The remainder of the paper is devoted to the proofs of Theorem~\ref{ThCP} and Theorem~\ref{ThCPQ}.

\subsection*{Proof of Theorem \ref{ThCP}.}

It is straightforward to check that every graph in $\mathcal{C}_1^{N}\cup \mathcal{C}_2^{N}\cup \mathcal{C}_1^{NP}\cup \mathcal{C}_2^{NP}$ is 2-connected, $c$-closed, claw-free and $N$-$p$-heavy. Let, conversely, $G$ be a 2-connected, $c$-closed, claw-free and $N$-$p$-heavy graph.

If $G$ is net-free, then $G\in \mathcal{C}^N_1\cup \mathcal{C}^N_2$ by Theorem~\ref{Ryjacek'N}. We assume now that $G$ contains an induced net, which is $p$-heavy. Let $K$ be a maximal clique containing all heavy vertices of $G$.
If $K$ does not contain all $a$-heavy pairs of $G$, there exists a maximal clique $K'$ such that $K$ intersects $K'$ at a heavy vertex, say $u_0$, and $K\cup K'$ contains all $a$-heavy pairs of $G$ by Lemma~\ref{LEHeavy}. If $K$ contains all $a$-heavy pairs of $G$, then we take $K'=\emptyset$ in the below analysis.

Let $H$ be a component of $G-(K\cup K')$. First we suppose that $N(H)\subseteq K$. It follows that $G[K\cup V(H)]$ is 2-connected; otherwise, there would exist a cut vertex in $G[K\cup V(H)]$, and this would also be a cut vertex in $G$, a contradiction. We shall show that $G[K\cup V(H)]\in \mathcal{C}_1^N\cup \mathcal{C}_2^N$. To see this, consider whether $G[K\cup V(H)]$ is net-free or not. Suppose first that $G[K\cup V(H)]$ contains an induced net $M$. Then $M$ is $p$-heavy. We label the vertices of $M$ as in Figure~\ref{N} such that $d(a_1)+d(a_2)\geq n$. Then $a_1,a_2\in K$.  Since $G[K]$ is a region of $G$, we have $a_3\in K$ by Lemma~\ref{LeRegion} (3).  It follows that $b_1,b_2,b_3\in V(H)$ and $b_ib_j\notin E(G)$ for $1\leq i\neq\ j \leq 3$. By Lemma~\ref{LeGeneralized}, there exists an induced generalized claw or net $M_0$ in $H$ connecting $b_1,b_2,b_3$. Let $Q_1,Q_2,Q_3$ be three paths of $M_0$ from the center or triangle to $b_1,b_2,b_3$, respectively. Since $b_1,b_2,b_3\in V(H)$ are three distinct vertices and  $b_ib_j\notin E(G)$ for $1\leq i\neq\ j \leq 3$, at most one of $Q_1,Q_2,Q_3$ is trivial. We assume that $Q_1$ is the shortest path in $Q_1,Q_2,Q_3$. Let $a_2'$ be a vertex in $K$ with $N(a_2')\cap V(H)=\emptyset$. This implies $G[V(Q_1)\cup V(Q_2)\cup V(Q_3)\cup\{a_1,a_2'\}]$ contains an induced generalized net. It follows that $V(H)\cup \{a_1\}$ contains an $a$-heavy vertex of $G$, contradicting the choice of $K$ or $K'$. Thus $G[K\cup V(H)]$ is net-free. Then $G[K\cup V(H)]\in \mathcal{C}_1^N\cup  \mathcal{C}_2^N$ by Theorem~\ref{Ryjacek'N}.

Next we suppose that $N(H)\subseteq K'$. It follows that $G[K'\cup V(H)]$ is 2-connected; otherwise, there would exist a cut vertex in $G[K'\cup V(H)]$, and this is also a cut vertex in $G$, a contradiction. We shall show that $G[K'\cup V(H)]\in \mathcal{C}_1^N\cup \mathcal{C}_2^N$. To see this, consider whether $G[K'\cup V(H)]$ is net-free or not. If $G[K'\cup V(H)]$ contains an induced net $M$, then it is $p$-heavy. Note that $u_0\notin N(H)$. Hence the triangle of $M$ does not contain $u_0$. Thus $K'\cup V(H)\setminus\{u_0\}$ contains a heavy vertex, contradicting the choice of $K$. We conclude that $G[K'\cup V(H)]$ is net-free, and then $G[K'\cup V(H)]\in \mathcal{C}_1^N\cup \mathcal{C}_2^N$ by Theorem~\ref{Ryjacek'N}.

Finally we suppose that $N(H)\cap K\neq \emptyset$ and $N(H)\cap K'\neq \emptyset$. It follows that $G[K'\cup V(H)]$ is 2-connected. We shall show that $G[K\cup K'\cup V(H)]\in \mathcal{C}_2^N$. To see this, consider whether $G[K\cup K'\cup V(H)]$ is net-free or not. If $G[K\cup K'\cup V(H)]$ contains an induced net $M$, then $M$ is $p$-heavy. We label the vertices of $M$ as in Figure~\ref{N} such that $d(a_1)+d(a_2)\geq n$. Then $a_1,a_2\in V(K)$ or $a_1,a_2\in V(K')$. By Lemma~\ref{LeRegion} (3), the triangle $a_1a_2a_3a_1$ is contained in $K$ or $K'$. First we consider that the triangle $a_1a_2a_3a_1$ is contained in $K$. Then
$b_1,b_2,b_3\in V(H)\cup (K'\setminus\{u_0\})$. By Lemma~\ref{LeGeneralized}, $V(H)\cup (K'\setminus \{u_0\})$ contains an induced generalized claw or net $M_1$.  Let $Q_1,Q_2,Q_3$ be three paths of $M_1$ from the center or triangle to $b_1,b_2,b_3$, respectively. Since $b_1,b_2,b_3\in V(H)\cup K'\setminus\{u_0\}$ are three distinct vertices and  $b_ib_j\notin E(G)$ for $1\leq i\neq\ j \leq 3$, at most one of $Q_1,Q_2,Q_3$ is trivial. We assume that $Q_1$ is the shortest path in $Q_1,Q_2,Q_3$. Let $a_2'$ be a vertex in $K$ with $N(a_2')\cap H=\emptyset$. This implies $G[V(Q_1)\cup V(Q_2)\cup V(Q_3)\cup\{a_1,a_2'\}]$ contains an induced generalized net. It follows that $V(H)\cup (K'\setminus \{u_0\})\cup \{a_1\}$ contains an $a$-heavy vertex of $G$, contradicting the choice of $K$ or $K'$. For $\{a_1,a_2,a_3\}\subset K'$, we can obtain a contradiction by using a similar analysis. We conclude that $G[K\cup K'\cup V(H)]$ is net-free. Thus $G[K\cup K'\cup V(H)]\in \mathcal{C}_2^N$.

Suppose first that $K'=\emptyset$. If $G-K$ has exactly one component $H$, then $G[K\cup V(H)]\in \mathcal{C}_1^N\cup \mathcal{C}_2^N$. If $G-K$ has at most two components and each of these components satisfies $G[K\cup V(H)]\in \mathcal{C}_1^N$, then $G\in \mathcal{C}_1^N$.
For the rest cases, it is clear that $G[K]$ has at least three frontier vertices. Let $u_1,u_2,u_3$ be three distinct frontier vertices of $G[K]$, and let $v_i$, $i=1,2,3$, be a neighbor of $u_i$ outside $K$. If $\{v_1,v_2,v_3\}$ is independent in $G$, then $G[\{u_1,u_2,u_3,v_1,v_2,v_3\}]$ is an induced net, which is $p$-heavy. So $\{u_1,u_2,u_3\}$ contains an $a$-heavy pair. Thus $G\in \mathcal{C}_1^{NP}$.

Now suppose that $K'\neq\emptyset$. It follows that there exists a component $H_0$ of $G-(K\cup K')$ satisfying $N(H_0)\cap K\neq \emptyset$ and $N(H_0)\cap K'\neq \emptyset$; otherwise $u_0$ is a cut vertex of $G$. Thus $G[V(H_0)\cup K\cup K']\in \mathcal{C}_2^N$.

If all components other than $H_0$ of $G-(K\cup K')$ satisfy $N(H)\cap K'= \emptyset$, then $G\in\mathcal{C}_1^{NP}$. Thus we assume that there exists another component $H_1$ of $G-(K\cup K')$ such that $N(H_1)\cap K'\neq \emptyset$.

If there is a third component having a neighbor in $K'$, or $N(H_1)\subset K'$ and $G[K'\cup V(H_1)]\in \mathcal{C}_2^N$, then there exists an induced net with its triangle contained in $K'\setminus\{u_0\}$. Thus $K'\setminus\{u_0\}$ contains a heavy vertex, a contradiction.

Let $u_1,u_2$ be two distinct frontier vertices other than $u_0$ of $G[K]$, and let $v_i$, $i=0,1,2$, be a neighbor of $u_i$ outside $K$. If $\{v_0,v_1,v_2\}$ is independent in $G$, then $G[\{u_0,u_1,u_2,v_0,v_1,v_2\}]$ is an induced net, which is $p$-heavy. So $\{u_0,u_1,u_2\}$ contains an $a$-heavy pair.

Let $x_1,x_2$ be two distinct frontier vertices other than $u_0$ of $G[K']$, and let $y_0, y_1, y_2$ be neighbors of $u_0, x_1, x_2$ outside $K'$, respectively. If {$\{y_0, y_1, y_2\}$ }is independent in $G$, then $G[\{u_0,x_1,x_2,y_0,y_1,y_2\}]$ is an induced net, which is $p$-heavy. So $\{u_0,x_1,x_2\}$ contains an $a$-heavy pair.
Thus $G\in  \mathcal{C}_2^{NP}$.

This completes the proof. \qed

\subsection*{Proof of Theorem \ref{ThCPQ}.}

The sufficiency follows obviously from Constructions \ref{CoC_1^NPQ} and \ref{CoC_2^NPQ}. Thus we shall show the necessity.  Let $G$ be a 2-connected, $c$-closed, claw-free and $N$-$pq$-heavy graph.

Let $K$ be a maximal clique containing all heavy vertices of $G$. If $G$ is $N$-$p$-heavy, then $G\in \mathcal{C}^N_1\cup
\mathcal{C}^N_2 \cup
\mathcal{C}^{NP}_1\cup
\mathcal{C}^{NP}_2$ by Theorem~\ref{ThCP}. Thus we assume that $G$ contains an induced $q$-heavy net $M_0$.
We label $M_0$  as in Figure~\ref{N} such that $a_1,b_1$ are heavy in $G$, $d(a_2)=d(a_3)=3$ and $d(b_2)=d(b_3)=2$. Let $c_2, c_3$ be neighbors of $b_2, b_3$ other than $a_2, a_3$, respectively. Then $c_2, c_3$ are heavy in $G$. Thus $a_1,b_1,c_2, c_3$ are contained in $K$. Since $N(a_2)=\{a_1,a_3,b_2\}$, we have $a_2c_2\notin E(G)$. The two maximal cliques containing $c_2$ are $K$ and $\{c_2,b_2\}$. Since $c_2$ is heavy, we have $|K|\geq \frac{n}{2}$.
In particular,
every frontier vertex of $G[K]$ is heavy in $G$.

If $K$ does not contain all $a$-heavy pairs of $G$, there exists a maximal clique $K'$ such that $K$ intersects $K'$ at a heavy vertex, say $u_0$, and $K\cup K'$ contains all $a$-heavy pairs of $G$ by Lemma~\ref{LEHeavy}. If $K$ contains all $a$-heavy pairs of $G$, then we take $K'=\emptyset$ in the below analysis.

Let $H$ be a component of $G-(K\cup K')$. First we suppose that $N(H)\subseteq K$. It follows that $G[K\cup V(H)]$ is 2-connected; otherwise, there would exist a cut vertex in $G[K\cup V(H)]$, and this would also be a cut vertex in $G$, a contradiction. We shall show that $G[K\cup V(H)]\in \mathcal{C}_1^N\cup \mathcal{C}_2^N\cup  \mathcal{C}_3^{NQ}$. To see this, consider whether $G[K\cup V(H)]$ is net-free or not. Suppose first that $G[K\cup V(H)]$ contains an induced net $M$. Then $M$ is $p$-heavy or $q$-heavy. We label $M$ as in Figure~\ref{N}. If $M$ is $p$-heavy, we can obtain a contradiction by using the same method as in the proof of Theorem~\ref{ThCP}. Thus $M$ is $q$-heavy. Without loss of generality, we assume that $a_1,b_1$ are heavy in $G$. Then $a_1,b_1$ are heavy in $G$, $d(a_2)=d(a_3)=3$ and $d(b_2)=d(b_3)=2$. It follows that $a_1,b_1\in K$ and $a_2,a_3,b_2,b_3\notin K$. Let $c_2$ and $c_3$ be the neighbors of $b_2$ and $b_3$ other than $a_2$ and $a_3$, respectively. Then $c_2,c_3$ are heavy in $G$. Thus $c_2,c_3\in K$. Thus $G[K\cup V(H)]\in \mathcal{C}_3^{NQ}$. If $G[K\cup V(H)]$ is net-free, then $G[K\cup V(H)]\in \mathcal{C}_1^N\cup \mathcal{C}_2^N$ by Theorem~\ref{Ryjacek'N}.

Next we suppose that $N(H)\subseteq K'$. It follows that $G[K'\cup V(H)]$ is 2-connected; otherwise, there would exist a cut vertex in $G[K'\cup V(H)]$, and this would also be a cut vertex in $G$, a contradiction. We shall show that $G[K'\cup V(H)]\in \mathcal{C}_1^N\cup \mathcal{C}_2^N$. To see this, consider whether $G[K'\cup V(H)]$ is net-free or not. If $G[K'\cup V(H)]$ contains an induced net $M$, then $M$ is $p$-heavy or $q$-heavy. If $M$ is $p$-heavy, we can obtain a contradiction using the same approach  as in the proof of Theorem~\ref{ThCP}. We omit the details. Thus $M$ is $q$-heavy, and $M$ contains two distinct heavy vertices of $G$. This implies that $K'\cup V(H)$ contains four distinct heavy vertices, a contradiction. Thus $G[K'\cup V(H)]$ is net-free. Then $G[K'\cup V(H)]\in \mathcal{C}_1^N\cup \mathcal{C}_2^N$ by Theorem~\ref{Ryjacek'N}.

Finally we suppose that $N(H)\cap K\neq \emptyset$ and $N(H)\cap K'\neq \emptyset$. It follows that $G[K'\cup V(H)]$ is 2-connected. We shall show that $G[K\cup K'\cup V(H)]\in \mathcal{C}_2^N$. To see this, consider whether $G[K\cup K'\cup V(H)]$ is net-free or not. If $G[K\cup K'\cup V(H)]$ contains an induced net $M$, then $M$ is $p$-heavy or $q$-heavy. If $M$ is $p$-heavy, we can obtain a contradiction by using the same method as in the proof of Theorem~\ref{ThCP}. Thus $M$ is $q$-heavy. We label $M$ as in Figure~\ref{N}. We assume that $a_1,b_1$ are heavy. Then $a_1,b_1\in K$, $d(a_2)=d(a_3)=3$ and $d(b_2)=d(b_3)=2$. It follows that $a_2,a_3,b_2,b_3\notin K$. Let $c_2,c_3$ be the neighbors of $b_2,b_3$ other than $a_2,a_3$, respectively. Then $c_2,c_3$ are heavy in $G$. Thus $c_2,c_3\in K$. It follows that there exists no path connecting $b_2$ to $K'$ in $H$, a contradiction. Thus $G[K\cup K'\cup V(H)]$ is net-free. Then $G[K\cup K'\cup V(H)]\in \mathcal{C}_2^N$ by Theorem~\ref{ThCPCZN}.

Recall that the induced net $M_0$ of $G$ is $q$-heavy but not $p$-heavy. If the triangle of $M_0$ is contained in $K'$, then $|K'|=3$ by the definition of $q$-heavy nets. Notice that $K'$ contains an $a$-heavy pair, implying that $M_0$ is $p$-heavy, a contradiction. Hence we conclude that the triangle of $M_0$ is not contained in $K'$. Thus there exists {a}
component $H$ of $G-(K\cup K')$ such that $N(H)\subset K$ and $G[V(H)\cup K]\in \mathcal{C}_3^{NQ}$.

If $K'=\emptyset$, then $G\in \mathcal{C}_1^{NPQ}$. Now suppose that $K'\neq\emptyset$. It follows that there exists a component $H_1$ of $G-(K\cup K')$ satisfying $N(H_1)\cap K\neq \emptyset$ and $N(H_1)\cap K'\neq \emptyset$, and $G[V(H_1)\cup K\cup K']\in \mathcal{C}_2^N$.

If all components other than $H_1$ of $G-(K\cup K')$ satisfy $N(H)\cap K'= \emptyset$, then $G\in\mathcal{C}_1^{NPQ}$. Thus we assume that there exists another component $H_2$ of $G-(K\cup K')$ such that $N(H_2)\cap K'\neq \emptyset$.

If there is a third component having a neighbor in $K'$, or $N(H_2)\subset K'$ and $G[K'\cup V(H_2)]\in \mathcal{C}_2^N$, then there exists an induced net with its triangle contained in $K'\setminus\{u_0\}$. Hence, $K'\setminus\{u_0\}$ contains a heavy vertex, a contradiction. Thus either $N(H_2)\subset K'$ and $G[K'\cup V(H_2)]\in \mathcal{C}_1^N$, or $N(H_2)\cap K\neq\emptyset$ and $G[K\cup K'\cup V(H_2)]\in \mathcal{C}_2^N$.

Let $u_1,u_2$ be two distinct frontier vertices other than $u_0$ of $G[K']$, and let $v_i$, $i=0,1,2$, be a neighbor of $u_i$ outside $K'$. If $\{v_0,v_1,v_2\}$ is independent in $G$, then $M=G[\{u_0,u_1,u_2,v_0,v_1,v_2\}]$ is an induced net, which is $p$-heavy or $q$-heavy. So $M$ is $p$-heavy. Thus $\{u_0,u_1,u_2\}$ contains an $a$-heavy pair. Hence $G\in \mathcal{C}_2^{NPQ}$.

This completes the proof. \qed

\section*{Declaration of competing interest}
The authors declare that they have no known competing financial interests or personal relationships that could have appeared to influence the work reported in this paper.

\section*{Data availability}
No data was used for the research described in the article.

\end{document}